\documentclass{SCIYA2023enOL}

\usepackage{amssymb}
\usepackage{amsthm}
\usepackage{amsfonts}
\usepackage{amsmath}
\usepackage{color}
\usepackage{mathtools}
\usepackage{natbib}
\usepackage{bbm}            

\usepackage{hyperref}

\hypersetup{colorlinks=true,citecolor=black,citecolor=black}

\usepackage{graphics}
\usepackage{mathrsfs}
\usepackage{amscd}
\usepackage{comment}
\usepackage{color}
\usepackage{url}

\usepackage{bm}

\newcommand{\bN} {{\mathbb{N}}}

\newcommand{\bW} {{\mathbb{W}}}
\newcommand{\bZ} {{\mathbb{Z}}}

\newcommand{\cF} {{\mathcal{F}}}
\newcommand{\cG} {{\mathcal{G}}}
\newcommand{\mycH} {{\mathcal{H}}}
\newcommand{\cS} {{\mathcal{S}}}
\newcommand{\cT} {{\mathcal{T}}}

\renewcommand{\span} {{\rm span}}

\newcommand{\lm} {{\rm lm}}

\newcommand{\ord} {{\rm ord}}
\newcommand{\tdeg} {{\rm tdeg}}

\def\<#1>{\langle#1\rangle}

\newcommand{\vi} {{\bm i}}
\newcommand{\vD} {{\bm D}}
\newcommand{\vj} {{\bm j}}

\newcommand{\vx} {{\bm x}}
\newcommand{\vy} {{\bm y}}

\newcommand{\vs} {{\bm s}}

\def\Dx{\vD_{x}}
\def\DS{\vD_{S}}

\newcommand{\bbeta}{{\boldsymbol{\beta}}}
\newcommand{\aalpha}{{\boldsymbol{\alpha}}}
\def\ee{\mathrm{e}}

\newcommand{\lcm} {{\rm lcm}}

\def\hx{\hat{\vx}}

\newtheorem{cor}[theorem]{Corollary}
\newtheorem{observation}[theorem]{Observation}

\theoremstyle{definition}
\newtheorem{notation}[theorem]{Notation}

\theoremstyle{remark}

\def\bm#1{\mathchoice{\kern-.5pt\mathord{\text{\bfseries\itshape#1}}\kern+.25pt}
                      {\kern-.5pt\mathord{\text{\bfseries\itshape#1}}\kern+.25pt}
                      {\kern+.25pt\mathord{\text{\scriptsize\bfseries\itshape#1}}\kern+.5pt}
                      {\kern+.25pt\mathord{\text{\tiny\bfseries\itshape#1}}}\kern+.25pt}

\usepackage{xcolor}

\online

\begin{document}

\ensubject{fdsfd}
\ArticleType{ARTICLES}
\Year{2025}
\Month{December}%
\Vol{??}
\No{?}
\BeginPage{1} %
\DOI{Science China Mathematics Manuscript for review}

\title[]{Single-exponential bounds for diagonals \\ of D-finite power series}
{Single-exponential bounds for diagonals \\ of D-finite power series}

\author[1,2,$\ast$]{Shaoshi Chen}{schen@amss.ac.cn}
\author[3]{Frédéric Chyzak}{frederic.chyzak@inria.fr}
\author[4]{Pingchuan Ma}{pcma@amss.ac.cn}
\author[5]{Chaochao Zhu}{zccjsbz@amss.ac.cn}

\AuthorCitation{Chen S., Chyzak F., Ma P., Zhu C.}

\address[1]{KLMM, Academy of Mathematics and Systems Science, \\ Chinese Academy of Sciences, Beijing 100190, China}
\address[2]{School of Mathematical Sciences, \\ University of Chinese Academy of Sciences, Beijing 100049, China}
\address[3]{Inria, 91120 Palaiseau, France}
\address[4]{Tsinghua University High School, Beijing 100084, China}
\address[5]{School of Finance and Mathematics, \\ West Anhui University, Lu'an 237012, China}

\abstract{%
D-finite power series appear ubiquitously in combinatorics, number theory, and mathematical physics.
They satisfy systems of linear partial differential equations
whose solution spaces are finite-dimensional,
which makes them enjoy a lot of nice properties.
After attempts by others in the 1980s,
Lipshitz was the first
to prove that the class they form in the multivariate case
is closed under the operation of diagonal.
In particular, an earlier work by Gessel had addressed
the D-finiteness of the diagonals of multivariate rational power series.
In this paper,
we give another proof of Gessel's result that fixes a gap in his original proof,
while extending it to the full class of D-finite power series.
We also provide a single exponential bound on the degree and order of the defining differential equation
satisfied by the diagonal of a D-finite power series in terms of the degree and order of the input differential system.
}



\maketitle

\section{Introduction}\label{SECT:intro}

Diagonals of multivariate formal power series appear frequently in different areas:
diagonals of rational power series play an important role in enumerative combinatorics, especially the lattice paths enumeration
(see the books \cite{Stanley-1999-EC2, Mishna-2020-ACM, PemantleWilson-2024-ACSV, Melczer-2021-IAC}
and the survey~\cite{PemantleWilson-2008-TCE});
Christol's number-theoretic conjecture,
which predicts that globally bounded D-finite power series are diagonals of rational power series~\cite{Christol-1987-FHB},
remains largely open
(see the nice survey \cite{Christol-2015-DRF} by himself);
intensive studies on diagonals also appear in computer algebra with connection to mathematical physics
\cite{AbdelazizKoutschanMaillard-2020-CC,BostanChyzakHoeijPech-2011-EFG,BostanDumontSalvy-2017-ADW,BostanBoukraaMaillardWeil-2015-DRF}.

In these contexts, formal power series are commonly given implicitly
as solutions to either algebraic or (linear) differential equations,
and the corresponding diagonals also satisfy such equations.
This is in particular the case for \emph{D-finite power series}.
Recall that these series are defined (Definition~\ref{DEF:dfinite})
as multivariate formal power series in variables $x_1,\dots,x_n$
whose infinite set of higher-order partial derivatives generates
a finite-dimensional vector space over the field of rational functions in the variables.
D-finite power series were first introduced and studied
by Stanley in 1980 in the univariate case~\cite{Stanley1980} and later systematically investigated by Lipshitz in the multivariate
case~\cite{Lipshitz1988, Lipshitz1989}.
In the early 1980's, Gessel, Stanley, Zeilberger, and many combinatorists conjectured that the diagonal of a rational power series in
several variables is D-finite.
Zeilberger~\cite{Zeilberger1980} in 1980 and Gessel~\cite{Gessel1981} in 1981 independently claimed to have proved this conjecture.
Later, in 1988, Lipshitz \cite{Lipshitz1988} pointed out that both proofs were not complete and he used a different, elementary idea to prove
that D-finite power series are closed under taking diagonals,
so that, in particular, diagonals of rational power series are D-finite.
In parallel, Christol had used the finiteness of some De Rham cohomology to prove the result:
first under some regularity assumption of a Jacobian variety~\cite{Christol-1984-DFR};
then in full generality~\cite{Christol-1985-DFR,Christol-1988-DFR}.
In 1990, Zeilberger~\cite{Zeilberger1990} then completed his own proof with the theory of holonomic D-modules.
Later, Wu and Chen \cite{WuChen-2013-NDT}
provided a similar result for the case of bivariate rational functions as a follow-up of Gessel's work.

The problem we address in this paper is to bound
the degrees and orders of linear differential equations satisfied by the diagonal of a given series
in terms of degrees and orders of the given differential systems
that the series satisfies.
We view this as a crucial preliminary step to the computational complexity analysis of algorithms for computing diagonals,
and to the longer-term development of fast algorithms in a complexity-driven way.

Diagonals of multivariate series come in several flavors (see Definition~\ref{DEF:diag}):
first, \emph{primary diagonals} collapse just two variables;
next, \emph{complete diagonals} collapse all variables to a single one.

Starting with primary diagonals, we get a polynomial increase of the order and degree bounds
(Corollary~\ref{cor:bivariate-gessel}).
A naive iteration of primary diagonals (Section~\ref{sec:iterating-prim-diags})
would thus lead to double-exponential bounds for complete diagonals (Section~\ref{sec:iterating-prim-diags}).
Our first and main contribution is therefore to derive a \emph{single-exponential} bound
(Theorems \ref{THM:primary bound} and~\ref{THM:complete bound}).
Note however that in the bivariate case ($n=2$),
no iteration is necessary
so that the double-exponential bound is in fact just polynomial,
and the bounds of Corollary~\ref{cor:bivariate-gessel} are better
than those of Theorems \ref{THM:primary bound} and~\ref{THM:complete bound}.

After Lipshitz's work~\cite{Lipshitz1988},
the general belief was that the gaps in Gessel's proof do not seem easy to fill.
As a secondary contribution,
we however fully fix and generalize Gessel's proof~\cite{Gessel1981} by elaborating on his original proof strategy
(Theorems \ref{THM:diagonal-n} and~\ref{THM:diagonal-2}).
Because Gessel's approach does not need any change of variables,
as opposed to Lipshitz's,
it leads more directly to explicit filtrations,
from which we benefit in our bound estimates of the Lipshitz way.

It is worth comparing the bounds we obtained in this paper with the situation in positive characteristic.
In that context, a result by Furstenberg~\cite{Furstenberg-1967-AFF} and Deligne~\cite{Deligne-1984-ICE}
states that the diagonal of any algebraic function is algebraic.
A~quantitative version of this theorem by Adamczewski and Bell \cite{AdamczewskiBell-2013-DRA}
provides bounds on the algebraic degree of a diagonal
and on the maximal degree (height) of a polynomial equation is satisfies,
which, even in the case of the diagonal of a \emph{rational} function,
is doubly exponential of the form~$O(p^{n^n})$,
where $p$~is the characteristic and $n$~is the number of variables.
As our bounds are singly exponential and might be useful also in characteristic~$p$,
this is another instance of the phenomenon \cite{Salvy-2019-LDE}
that representing an algebraic function by differential equations
is more compact than by a polynomial equation.
The bound in~\cite{AdamczewskiBell-2013-DRA} has very recently been significantly improved in~\cite[Theorem~5.2]{AdamczewskiBostanCaruso-2023-SMC}.

In the case of characteristic zero,
the first bound on the order of an annihilator of the diagonal of a rational power series
was given by Christol~\cite{Christol-1984-DFR}, under a regularity assumption.
In~\cite{BostanLairezSalvy-2013-CTR}, single-exponential bounds were announced for both order and degree, still in the rational case.
Other single-exponential bounds have been announced
for differential operators cancelling Hadamard products of rational series (and therefore diagonals of rational series)
in the extended version~\cite{BostanCarayolKoechlinNicaud-2025-WUP} of a work~\cite{BostanCarayolKoechlinNicaud-2020-WUP} related with a theoretical study on automata:
this indicates the existence of an annihilating operator
satisfying single-exponential bounds on its order, its degree, as well as the height of its coefficients.
Our contribution can therefore be viewed as a generalization of these results on order and degree
to general D-finite power series.

The remainder of this paper is organized as follows.
We recall some basic terminology about rings of differential operators
and introduce D-finite power series and their diagonals in Section~\ref{SECT:dfinite}.
In Section~\ref{SECT:diagonal}, we first prove the Diagonal Theorem (Theorem~\ref{THM:diagonal-n}) on D-finite power series
in the way suggested in Gessel's work
and we then derive an explicit polynomial bound
for annihilators of diagonals in the bivariate case.
Then, a single-exponential bound is given for the general multivariate situation in Section~\ref{SEC:bound}
by analyzing Lipshitz's proof.

\section{Differential operators, D-finiteness, and diagonals}\label{SECT:dfinite}

Throughout this article, we assume that~$K$ is a field of characteristic~$0$.
Let~$K[\vx]$ be the ring of polynomials in $\vx = x_1, \ldots, x_n$
over $K$ and~$K(\vx)$~be the field of rational functions~in~$\vx$ over $K$.
Let~$K[[\vx]]$ be the ring of formal power series~in~$\vx$ over~$K$, which is a domain.
Denote $\cS := K(\vx) \otimes_{K[\vx]} K[[\vx]]$.
Let~$ D_{x_1}, \ldots, D_{x_n}$ denote the usual partial derivations~$\partial/\partial{x_1}, \ldots, \partial/\partial{x_n}$~on~$\cS$.
This is the basic notation that we will use in Sections \ref{SECT:intro} and~\ref{SECT:diagonal},
but it will need to be generalized in Section~\ref{SEC:bound}.

The \emph{Weyl algebra}~$\bW_n$ is the non-commutative polynomial ring in the variables
$\vx =x_1,\dots,x_n$ and $\Dx = D_{x_1}, \ldots, D_{x_n}$, in which the following multiplication rules hold:
$ x_i x_j = x_j x_i, D_{x_i} D_{x_j} = D_{x_j}D_{x_i}$ for all $i, j\in \{1, \ldots, n\}$
and~$D_{x_i} a = a D_{x_i} + \partial a/\partial x_i$ for all~$i\in \{1, \ldots, n\}$ and $a \in K[\vx]$.
We will also write~$K[\vx]\<\Dx>$ for~$\bW_n$,
as, here and throughout, we use angled brackets~$R\<\dots>$ to denote a twisted extension of a ring~$R$,
when generators between brackets always commute with one another.
The Weyl algebra can be interpreted as the ring of linear partial differential operators with
polynomial coefficients.
We will also use the ring $K(\vx)\<\Dx>$ of linear partial differential operators with \emph{rational} function coefficients.
The elements of this ring act on~$\cS$ by interpreting $D_{x_i}$ as $\partial/\partial x_i$, which turns $\cS$ into a left $K(\vx)\<\Dx>$-module.
For a given $f\in \cS$, the \emph{annihilating ideal} of~$f$ in~$K(\vx)\<\Dx>$ is defined as the set $\{L \in K(\vx)\<\Dx>\mid L(f) =0\}$.
Note that this is indeed a left $K(\vx)\<\Dx>$-module,
therefore in particular a vector space over~$K(\vx)$.

\begin{notation}
Given a polynomial~$a \in K[\vx]$ and an operator~$P$ (in~$K[\vx]\<\Dx>$ or in~$K(\vx)\<\Dx>$),
we will distinguish the expression~$P a$, with no parentheses, from the expression~$P(a)$, with parentheses:
the former will always denote the product in the operator algebra;
the latter will always denote application of the operator to~$a$, viewed as a series in~$\cS$.

In contrast, for a series~$f$ in~$\cS$, we will never have to denote a product,
and both $P f$ and~$P(f)$ will denote application.
\end{notation}

\begin{definition}[D-finiteness]\label{DEF:dfinite}
An element $f\in K[[\vx]]$ is \emph{D-finite} over $K(\vx)$
if the $K(\vx)$-vector space generated in~$\cS$ by the derivatives $D_{x_1}^{\alpha_1}\dotsm D_{x_n}^{\alpha_n}(f)$
when $\alpha_1,\dots,\alpha_n$ range over~$\bN$
is finite-dimensional.
\end{definition}

Note that $L(f)$ is also D-finite for any operator $L\in K(\vx)\<\Dx>$.

\begin{definition}[Order and degree]\label{DEF:bound}
Assume that $f\in \cS$ is D-finite over $K(\vx)$. Then for each $i\in \{1, \ldots, n\}$, there exists a non-zero operator $L_i $ in the subalgebra $ K[\vx]\<D_{x_i}>$ of~$ \bW_n $
such that $L_i (f) = 0$.
Write
\begin{equation}\label{eq:L_i}
L_i = \ell_{i,0} + \ell_{i,1}  D_{x_i} + \cdots + \ell_{i,r_i} D_{x_i}^{r_i}
\end{equation}
with $ \ell_{i,0}, \dots \ell_{i,r_i} \in  K[\vx] $ and $ \ell_{i,r_i} \neq 0 $. We call $ r_i $ the
\emph{order} of the operator $ L_i $, denoted by $ \ord (L_i) $. The \emph{degree} of $ L_i $ is defined as the
maximum total degree of its polynomial coefficients: $ \deg(L_i) := \max_{j=0}^{r_i} \tdeg(\ell_{i,j}) $, where $ \tdeg $ means the total degree with respect to $ x_1, \dots, x_n $. Let $ r_f := \max_{i=1}^{n} r_i $ and $ d_f := \max_{i=1}^{n} d_i $ where $ d_i = \deg(L_i) $.
\end{definition}

\begin{definition}\label{DEF:diag}
Let $f= \sum_{i_1,\ldots, i_n \geq 0} a_{i_1,\ldots, i_n} x_1^{i_1}\cdots x_n^{i_n} \in K[[\vx]]$.
We call the power series
\begin{equation*}
\Delta_{1,2}(f) := \sum_{i_1,i_3, \ldots, i_n \geq 0} a_{i_1,i_1,i_3,\ldots, i_n} x_1^{i_1} x_3^{i_3}\cdots x_n^{i_n} \in K[[x_1, x_3,\dots,x_n]]
\end{equation*}
a \emph{primary diagonal} of~$f$.
Other primary diagonals~$\Delta_{i,j}$ are defined similarly,
so that $\Delta_{i,j}(f)$ and~$\Delta_{j,i}(f)$ are the same series
except for the variable names.
A~\emph{diagonal} is defined as any composition of the~$\Delta_{i,j}$.
The \emph{complete diagonal} of~$f$, denoted by~$\Delta(f)$, is defined as
\begin{equation}\label{eq:complete-diag}
\Delta(f) := \Delta_{n,n-1}\Delta_{n-1,n-2} \cdots \Delta_{2,1}(f) = \sum_{i\geq 0} a_{i, \dots, i}x_n^i \in K[[x_n]] .
\end{equation}
By \emph{the} diagonal of~$f$, we mean its complete diagonal
when no ambiguity arises.
\end{definition}

For future reference, we recall here the following well-known consequence of Cramer's rules
that will be used in the subsequent sections.
\begin{lemma}\label{LM:deg}
Let $A = (a_{i, j}) \in K[\vx]^{n \times m}$ be a matrix
with entries of total degree at most~$d$.
Assume the inequality~$n<m$,
so that the matrix has a non-trivial right nullspace.
Then, there exists a non-zero vector $v=(v_1, \ldots, v_m) \in K[\vx]^m$
that solves $A v = 0$ and has total degree at most~$nd$.
\end{lemma}

\begin{proof}
Let $\rho$ denote the rank of $A$.
Because $\rho \leq n$, we can fix $\rho$ linearly independent rows of $A$
and form a $\rho\times m$ submatrix $B$ of $A$ of rank $\rho$.
In turn, consider $\rho$ linearly independent columns of $B$,
thus forming a $\rho\times \rho$ submatrix~$C$,
and an additional column $c$ of $B$.
The system $Cw = -c$ admits a non-zero solution $w$
with $\tdeg (w) \leq \rho d$
that can be expressed by Cramer's rules.
Padding $w$ with zeros results in a non-zero $v$
satisfying $Av=0$ and $\tdeg (v) \leq \rho d \leq n d$ as wanted.
\end{proof}

\section{Diagonal theorem in the multivariate case}\label{SECT:diagonal}
In this section, we give a proof of the following “Diagonal theorem” in the spirit of Gessel~\cite{Gessel1981}.

\begin{theorem}[Diagonal Theorem]\label{THM:diagonal-n}
Let $f\in K[[\vx]]$ be D-finite over $K(\vx)$.
Then $\Delta(f)\in K[[x_n]]$ is D-finite over~$K(x_n)$.
\end{theorem}

The proof of Theorem~\ref{THM:diagonal-n} is just an iteration of the following result
for primary diagonals.

\begin{theorem}\label{THM:diagonal-2}
	Let $f\in K[[\vx]]$ be D-finite over $K(\vx)$. Then $\Delta_{1,2}(f)$
	is D-finite over $K(x_1, x_3,\dots,x_n)$.
\end{theorem}

The rest of the present section is devoted to the proof of Theorem~\ref{THM:diagonal-2}.

\medskip

The following objects will serve as generators in relevant algebras:
\begin{equation*}
\begin{aligned}
D_{x_1, x_2} &:= D_{x_1}D_{x_2} , \\
\theta_{x_i} &:= x_i D_{x_i} \quad\text{for each $i \in \lbrace 1,\dots,n \rbrace$,} \\
T_{x_1, x_2} &:= \theta_{x_1}-\theta_{x_2} .
\end{aligned}
\end{equation*}
We use bold notation to abbreviate monomials:
for example,
$\vx^\aalpha$~denotes $x_1^{\alpha_1} \dots x_n^{\alpha_n}$ and $\vD_x^\bbeta$~denotes $D_{x_1}^{\beta_1}\dots D_{x_n}^{\beta_n}$.
By ~\cite[Proposition 2.1]{Coutinho1995}, the set $\{\vx^\aalpha  \vD_x^\bbeta \mid \aalpha, \bbeta \in \bN^n\}$ is a basis of $\bW_n = K[\vx]\<\vD_x>$ as a vector space over~$K$.
Similarly, Lemmas \ref{LM:basis}, \ref{LM:basis-2}, and~\ref{LM:basis-3}
are lemmas providing canonical bases for several subalgebras of~$\bW_n$:
$K[x_1 x_2, x_3, \dots, x_n]\<T_{x_1, x_2}, D_{x_1, x_2}>$,
$K[x_1, x_3, \dots, x_n]\<D_{x_1}>$,
and $K[x_1 x_2, x_3, \dots, x_n]\<T_{x_1, x_2}, D_{x_h}>$
for $h \in \{3, \dots, n\}$.

\begin{lemma} \label{LM:basis}
The set
\begin{equation*}
\{(x_1 x_2)^i x_3^{k_3} \cdots x_n^{k_n} T_{x_1, x_2}^j D_{x_1, x_2}^{\ell} \mid i, j, \ell, k_3,\ldots, k_n \in \bN\}
\end{equation*}
is a basis of $K[x_1 x_2, x_3, \dots, x_n]\<T_{x_1, x_2}, D_{x_1, x_2}>$ as a vector space over~$K$.
\end{lemma}

\begin{proof}
It suffices to show that the monomials $(x_1 x_2)^i T_{x_1, x_2}^j D_{x_1, x_2}^{\ell}$ are linearly independent over~$K$.
Suppose that
\begin{equation*}
L = \sum_{(i,j,\ell) \in \Lambda} c_{i, j, \ell}\ (x_1 x_2)^i T_{x_1, x_2}^j D_{x_1, x_2}^{\ell} = 0
\end{equation*}
for some non-empty finite set $ \Lambda $ and $c_{i, j, \ell} \in K \setminus \{0\}$.
Let $\succ$ be the lexicographical order on the algebra $\bW_2 = K[x_1,x_2]\<D_{x_1}, D_{x_2}>$ with $D_{x_1} \succ D_{x_2} \succ x_2 \succ x_1$.
For any non-zero element~$Q$ of~$\bW_2$,
we write~$\lm(Q)$ to denote its leading monomial,
that is, the highest monomial with respect to~$\succ$ occurring in~$Q$ with a non-zero coefficient.
It can be proved by induction that there exist $Q_1$ and~$Q_2$ in~$\bW_2$ such that
\begin{equation*}
T^j_{x_1, x_2} = \theta_{x_1}^j + Q_1
  = x_1^j D_{x_1}^j  + Q_2 \quad \text{and} \quad \lm (Q_1) < \lm (\theta_{x_1}^j), \ \lm (Q_2) < x_1^j D_{x_1}^j .
\end{equation*}
Then we have that there exists~$Q_3$ in~$\bW_2$ such that
\begin{equation*}
(x_1 x_2)^i T_{x_1, x_2}^j D_{x_1, x_2}^{\ell} = x_1^{i+j}  x_2^i D_{x_1}^{j+\ell}D_{x_2}^\ell  + Q_3 \quad \text{and} \quad \lm (Q_3) < x_1^{i+j}  x_2^i D_{x_1}^{j+\ell}D_{x_2}^\ell .
\end{equation*}
Note that the map $(i,j,\ell) \mapsto (i+j,i, j+\ell, \ell)$ is injective.
Since the set $\{\vx^\vi \Dx^\vj \mid \vi, \vj \in \bN^n\}$ is a basis of $\bW_n = K[\vx]\<\Dx>$ as a vector space over $K$, this forces all $ c_{i,j, \ell} = 0 $, which contradicts our assumption.
\end{proof}

\begin{lemma} \label{LM:basis-2}
	The set $$\{x_1^k x_3^{k_3} \cdots x_n^{k_n} D_{x_1}^{\ell} \mid k,  k_3,\ldots, k_n, \ell\in \bN\}$$ is a basis of $K[x_1, x_3, \ldots, x_n]\<D_{x_1}>$ as a vector space over~$K$.
\end{lemma}

\begin{lemma} \label{LM:basis-3}
	For each $ h \in \{ 3,\dots, n \} $, the set $$\{(x_1 x_2)^i x_3^{k_3} \cdots x_n^{k_n} T_{x_1, x_2}^j D_{x_h}^{\ell} \mid i, j, \ell, k_3,\ldots, k_n \in \bN\}$$ is a basis of $K[x_1 x_2, x_3, \dots, x_n]\<T_{x_1, x_2}, D_{x_h}>$ as a vector space over~$K$.
\end{lemma}

We omit the proofs of Lemma \ref{LM:basis-2} and \ref{LM:basis-3} that are very similar to the proof of Lemma \ref{LM:basis}.
Next we present some commutation rules between the diagonal operator $\Delta_{1,2}$ and the operators~$x_1x_2, D_{x_1, x_2}, \theta_{x_i}$ and $T_{x_1, x_2}$.

\begin{proposition}\label{PROP:diagonal}
For any power series $f(\vx)\in K[[\vx]]$, we have
\begin{enumerate}
\item\label{IT:Delta-xx}
  $\Delta_{1,2}(x_1 x_2 f) = x_1 \Delta_{1,2}(f) $;
\item\label{IT:Delta-Dxx}
  $\Delta_{1,2}(D_{x_1, x_2}(f)) = D_{x_1}\theta_{x_1}(\Delta_{12}(f)) $;
\item\label{IT:Delta-theta1}
  $\Delta_{1,2}(\theta_{x_1}(f)) = \theta_{x_1}(\Delta_{1,2}(f)) $;
\item\label{IT:Delta-theta2}
  $\Delta_{1,2}(\theta_{x_2}(f)) = \theta_{x_1}(\Delta_{1,2}(f)) $;
\item\label{IT:Delta-Txx}
  $\Delta_{1,2}(T_{x_1, x_2}(f)) = 0 $;
\item\label{IT:Txx-Dxx}
  $D_{x_1, x_2}T_{x_1, x_2} = T_{x_1, x_2}D_{x_1, x_2}$;
\item\label{IT:Txx-xx}
  $T_{x_1, x_2} \ x_1 x_2 = x_1 x_2 \ T_{x_1, x_2}$.
\end{enumerate}
\end{proposition}

\begin{proof}
Given $f = \sum_{i_1,\dots, i_n \geq 0} a_{i_1,\dots, i_n} x_1^{i_1} \dotsm x_n^{i_n} \in K[[\vx]]$, we have
\begin{equation*}
\begin{aligned}
\Delta_{1,2}(x_1 x_2 f) &=
    \Delta_{1,2}\biggl(\sum_{i_1,\ldots, i_n \geq 0} a_{i_1,\ldots, i_n} x_1^{i_1+1} x_2^{i_2+1} x_3^{i_3} \dotsm x_n^{i_n}\biggr) \\
    &= \sum_{i_1, i_3, \ldots, i_n \geq 0} a_{i_1, i_1 \ldots, i_n} x_1^{i_1+1} x_3^{i_3}\cdots x_n^{i_n} = x_1 \Delta_{1,2}(f) ,
\end{aligned}
\end{equation*}
which proves Point~\ref{IT:Delta-xx}.
Points \ref{IT:Delta-Dxx}, \ref{IT:Delta-theta1}, and~\ref{IT:Delta-theta2}
are proved in \cite[Lemma~4.3]{WuChen-2013-NDT}.
Point~\ref{IT:Delta-Txx} immediately follows by linearity from Points \ref{IT:Delta-theta1} and~\ref{IT:Delta-theta2}.
Taking the difference of the two identities
\begin{equation*}
\begin{aligned}
D_{x_1, x_2} (x_1 D_{x_1}) &= D_{x_1}(x_1 D_{x_1})D_{x_2} = (x_1 D_{x_1}+1) D_{x_1, x_2} , \\
D_{x_1, x_2} (x_2 D_{x_2}) &= D_{x_2}(x_2 D_{x_2})D_{x_1} = (x_2 D_{x_2}+1) D_{x_1, x_2} ,
\end{aligned}
\end{equation*}
we obtain Point~\ref{IT:Txx-Dxx}.
Similarly, taking the difference of the two identities
\begin{equation*}
\begin{aligned}
(x_1 D_{x_1}) \ x_1 x_2 &= x_1 (x_1 D_{x_1} + 1) x_2 = x_1 x_2 \ (x_1 D_{x_1}+1) , \\
(x_2 D_{x_2}) \ x_1 x_2 &= x_2 (x_2 D_{x_2} + 1) x_1 = x_1 x_2 \ (x_2 D_{x_2}+1) ,
\end{aligned}
\end{equation*}
proves Point~\ref{IT:Txx-xx}.
\end{proof}

\begin{lemma}\label{LEM:xy}
Let $f(\vx)\in K[[\vx]]$.
Then, there exists $s \in \bN$ such that $T_{x_1, x_2}^s(f)=0$ if and only if there exists $g$ in $n-1$ variables such that $f(\vx) = g(x_1x_2,x_3,\dots,x_n)$.
\end{lemma}
\begin{proof}
	If $f(\vx) = g(x_1x_2,x_3,\dots,x_n)$, write
	$$ g(x_1, x_3, \dots, x_n) = \sum_{i_1, i_3,\dots, i_n\geq 0} b_{i_1, i_3,\ldots, i_n} x_1^{i_1} x_3^{i_3}\cdots x_n^{i_n}. $$
	Take $ s = 1 $, then
	$$
	\begin{aligned}
		T_{x_1, x_2}(f) &= T_{x_1, x_2} (g(x_1x_2,x_3,\dots,x_n)) \\ &= \sum_{i_1, i_3,\dots, i_n\geq 0} \, (i_1 - i_1) \, b_{i_1, i_3, \ldots, i_n} \, x_1^{i_1} x_2^{i_1} x_3^{i_3}\cdots x_n^{i_n} = 0.
	\end{aligned}
	$$
	For the converse statement, assume there exists $ s \in \bN $ such that $T_{x_1, x_2}^s(f)=0$. If $ s = 0 $, then $ f = 0 $. Take $ g=0 $. If $ s>0 $, write $f = \sum_{i_1, \dots, i_n\geq 0} a_{i_1,\ldots, i_n} x_1^{i_1}\cdots x_n^{i_n} $. Then $$T_{x_1, x_2}^s(f) = \sum_{i_1, \dots, i_n\geq 0} (i_1-i_2)^s a_{i_1,\ldots, i_n} x_1^{i_1}\cdots x_n^{i_n} = 0.$$
	Hence $(i_1-i_2)^s a_{i_1,\ldots, i_n}=0$ for all integers $i_1, \dots, i_n \geq 0$, so that $a_{i_1,\ldots, i_n}=0$ for all $i_1\neq i_2$. Take
	$$ g(x_1, x_3, \dots, x_n) = \sum_{i_1, i_3, \dots, i_n\geq 0} a_{i_1, i_1, i_3,\ldots, i_n} x_1^{i_1} x_3^{i_3}\cdots x_n^{i_n}. $$ Then $ f(\vx) = g(x_1x_2,x_3,\dots,x_n) $.
\end{proof}

\begin{lemma}\label{LEM:dfinite-alg}
Let $f(\vx)\in K[[\vx]]$ be D-finite over~$K(\vx)$.
Write $\vy$ for $y_1,\dots, y_m$
and consider power series
\begin{equation*}
g_1(\vy), \dots, g_n(\vy) \in K[[\vy]]
\end{equation*}
that are algebraic over~$K(\vy)$.
Assume that the substitution $f(g_1(\vy), \dots, g_n(\vy))$ is well-defined in~$K[[\vy]]$.
Then the series $f(g_1(\vy), \dots, g_n(\vy))$ is D-finite over~$K(\vy)$.
In particular, let $f(\vx)$ be D-finite over~$K(\vx)$
and suppose that the evaluation of $f(\vx)$ at $x_2=1$
is well-defined as a series in~$K[[x_1,x_3,\dots,x_n]]$,
then $f(x_1,1,x_3,\dots, x_n)$ is D-finite
over~$K(x_1,x_3,\dots,x_n)$.
\end{lemma}
\begin{proof}
	See \cite[Proposition 2.3]{Lipshitz1989}.
\end{proof}
From the definition of $K(\vx)\<\Dx>$,
\[D_{x_i} a = a D_{x_i} + D_{x_i}(a) \quad \text{for all $a \in K(\vx)$}.\]
More generally, we have the formulae: for all~$a \in K(\vx)$ and $i\in \{1,\dots,n\}$,
\begin{equation} \label{EQ:leibniz}
	D_{x_i}^k a=\sum_{\ell =0}^k{k \choose \ell} D_{x_i}^{\ell}(a)D_{x_i}^{k-\ell}
\end{equation}
and
\begin{equation} \label{EQ:leibniz_dual}
	a D_{x_i}^k= \sum_{\ell=0}^k (-1)^{\ell}{k \choose \ell} D_{x_i}^{k-\ell} D_{x_i}^{\ell}(a).
\end{equation}
The relations~\eqref{EQ:leibniz} and \eqref{EQ:leibniz_dual} can be proved by a straightforward induction.
In the sequel, we merely use the facts that,
for all~$a \in K[\vx]$ and $i\in \{1,\dots,n\}$,
\begin{equation} \label{EQ:lin}
	D_{x_i}^k a = f D_{x_i}^k + P  \quad \quad {\rm and} \quad \quad a D_{x_i}^k = D_{x_i}^k a - P,
\end{equation}
where~$P \in K[\vx]\<D_{x_i}>$ with $\ord(P) < k$ and $\deg(P) \le \deg(a)$.
Denote by $ K[\vx]_{\leq d} $ the set of polynomials in $ K[\vx] $ with total degree less than or equal to~$d$.

A number of similar arguments in the rest of the article will differ
only by a choice of variables.
This is why we make some notation depend on a set~$S$ in the following definition.
The reader is invited to pay attention to this implicit dependency in what follows.
We will indeed use $S = \{1,2\}$ and~$S = \{1,2,h\}$ in the present Section~\ref{SECT:diagonal},
and we will additionally use $S = \{1,\dots,n\}$ in Section~\ref{SEC:bound}.

\begin{notation}\label{DEF:C}
Fix a subset $ S \subseteq \{ 1,2, \dots,n  \}  $.
Let $L_i \in \bW_n$, for~$1\leq i\leq n$, be operators defined by~\eqref{eq:L_i} as in Definition~\ref{DEF:bound}.
In particular, recall $L_i = \ell_{i,r_i} D_{x_i}^{r_i} + \dots$ for some non-zero polynomial~$\ell_{i,r_i} \in K[\vx]$.
We give the following definitions and notation:
\begin{enumerate}
\item Given $ \bbeta \in \bN^S $, write $ \DS^\bbeta $ for the product $ \prod_{j \in S} D_{x_j}^{\beta_j} $,
\item $C := \lcm_{j \in S} (\ell_{j,r_j}) \in K[\vx]$,
\item for each $ j \in S $, $ \tilde{L}_j := (C / \ell_{j,r_j}) L_j \in \bW_n $,
\item $ d_C := \sum_{j\in S} d_j $,
\item $B:= \prod_{j \in S} \lbrace   0, 1, \dots r_i-1 \rbrace \subseteq \bN^{\#S}$,
\item $ \cF_{d,r} := \bigoplus_{ | \bbeta | \leq r} K[\vx]_{\leq d} \; \DS^\bbeta $,
\item $ \mycH_{d,r} := \bigoplus_{ | \bbeta | \leq r \, \text{or} \, \bbeta \in B } K[\vx]_{\leq d} \; \DS^\bbeta $,
\item $ J := \sum_{j\in S} \bW_n \tilde{L}_j $,
\end{enumerate}
where the dependency in~$S$ is kept implicit in the notation.
\end{notation}
Immediately we have
\begin{lemma}\label{LM:C}
For any non-empty set $ S \subseteq \{ 1,2, \dots,n  \} $, consider the quantities in Definition \ref{DEF:C}. Then
\begin{enumerate}
\item $\tdeg (C) \leq  d_C$,
\item for each $ i\in S $, $ \tilde{L}_i (f) = 0 $ and $ \deg \tilde{L}_i \le d_C $,
\item for each $ i \in S $, $ C D_{x_i}^{r_i} = (C D_{x_i}^{r_i} - \tilde{L}_i) + \tilde{L}_i \in \cF_{d_C,r_i-1} + J. $
\end{enumerate}
\end{lemma}
Continuing in analogy with \cite[Lemma~3]{Lipshitz1988}, we have the following lemmas:
\begin{lemma} \label{LM:reduce-1}
For all $\aalpha \in \bN^S$, $C \DS^\aalpha$~is an element of~$\cF_{d_C, | \aalpha |-1} + J$.
\end{lemma}

\begin{proof}
If $ \aalpha \in B $, nothing needs to be proven. So suppose, for instance, $ n \in S $ and $ \alpha_n \geq r_n $. Then multiply $ \DS^{\aalpha - r_n \ee_n} $ with $  C  D_{x_n}^{r_n} $, where $ e_n := (0,0,\dots,1) $, which yields
\begin{equation*}
\begin{aligned}
\DS^{\aalpha - r_n \ee_n}  C  D_{x_n}^{r_n} &\in \bigoplus_{j=0}^{r_n-1} \DS^{\aalpha - r_n \ee_n} K[\vx]_{\leq d_C} D_{x_n}^j + J \\
&\subseteq  \sum_{j=0}^{r_n-1} \left( K[\vx]_{\leq d_C} \DS^{\aalpha - r_n \ee_n} + \cF_{d_C, | \aalpha |-r_n-1}\right)  D_{x_n}^j + J \\
&\subseteq \cF_{d_C, | \aalpha |-1} + \cF_{d_C, | \aalpha |-2} + J  = \cF_{d_C, | \aalpha |-1} +J.
\end{aligned}
\end{equation*}
Hence
\begin{equation*}
\begin{aligned}
C \DS^\aalpha \in (\DS^{\aalpha - r_n \ee_n}  C + \cF_{d_C, | \aalpha |-r_n-1} ) D_{x_n}^{r_n} &\subseteq \DS^{\aalpha - r_n \ee_n}  C D_{x_n}^{r_n} + \cF_{d_C, | \aalpha |-1} \\
&\subseteq \cF_{d_C, | \aalpha |-1} + J.
\end{aligned}
\end{equation*}
\end{proof}

\begin{lemma} \label{LM:reduce-1.5}
	For any $ t \in \bN, r \in \bZ $, $ C\mycH_{t,r} \subseteq  \mycH_{d_C+t,r-1} + J $.
\end{lemma}

\begin{proof}
	We have the chain of equalities and inclusions:
	$$
	\begin{aligned}
		C\mycH_{t,r} &= C \bigoplus_{ | \bbeta | \leq r \, \text{or} \, \bbeta \in B } K[\vx]_{\leq t} \; \DS^\bbeta = \bigoplus_{ | \bbeta | \leq r \, \text{or} \, \bbeta \in B } K[\vx]_{\leq t} \; C \DS^\bbeta \\
		&\subseteq \sum_{\bbeta \in B} K[\vx]_{\leq d_C + t} \, \DS^\bbeta + \sum_{ | \bbeta | \leq r \, \text{and} \, \bbeta \notin B} K[\vx]_{\leq t} \, \cF_{d_C,| \bbeta | -1} + J \\
		&\subseteq \mycH_{d_C+t, r-1} + J,
	\end{aligned}
	$$
	where the first inclusion is by Lemma \ref{LM:reduce-1}.
\end{proof}

\begin{lemma} \label{LM:reduce-2}
	For any $ u, t \in \bN, v \in \bZ $, if $ u \geq v $, then $ C^{u} \mycH_{t,v} \subseteq  \mycH_{t+u d_C,0} + J $. In particular, for all $ \aalpha \in \bN^S $, $ C^{|\aalpha|} \DS^\aalpha \in \mycH_{ |\aalpha| d_C,0} + J $.
\end{lemma}

\begin{proof}
	Note that for all $ r' \leq 0 $, $ \mycH_{ t, r'} = \mycH_{t,0} $. The result is obtained by making~$ u $ repetitions of Lemma \ref{LM:reduce-1.5}. 
\end{proof}
Lemma \ref{LM:reduce-2} is specialized as follows.
\begin{lemma}
	Set $ u:= v + 1 - \min_{i\in S} r_i $. Then
	$ C^{u} \mycH_{t,v} \subseteq \mycH_{t+u d_C,0} +J $.
\end{lemma}

\begin{proof}
	Observe that for any $ \bbeta \in \bN^S $, if $ |\bbeta | < \min r_i $, then $ \bbeta \in B $. Hence for any $ r' < \min r_i $, $ \mycH_{ t, r'} = \mycH_{t,0} $. Again, the result is obtained by repeating the use of Lemma \ref{LM:reduce-1.5} $ u $ times.
\end{proof}

\begin{observation}\label{LM:ineq1}
	For positive integers $ D$ and $R $, define $ N = 3 D^2 R $, then
	$$ 		 \binom{N+3}{3} - R \binom{DN+2}{2}  > 0.  $$
\end{observation}

\begin{proof}
The result follows from the equality
\begin{equation*}
\binom{N+3}{3} - R \binom{DN+2}{2} = 9 R^2 D^3\left( D - \frac{1}{2}\right)  + R \left( \frac{11}{2}D^2-1\right) +1 .
\end{equation*}
\end{proof}

The following result provides structured annihilating operators of~$f$
whose existence will be used in the proof of Theorem~\ref{THM:diagonal-2}.
It also provides degree bounds for all the announced annihilating operators,
of which only those concerning~$P$ will be used,
in the specific situation of Corollary~\ref{cor:bivariate-gessel} ($n=2$).

\begin{theorem}\label{THM:elim}
Let $f\in K[[\vx]]$ be a D-finite power series over~$K(\vx)$.
Then, there exists a non-zero annihilating operator~$P$ of~$f$ that satisfies
\begin{itemize}
\item $P\in K[x_3, \dots, x_n][x_1 x_2]\<T_{x_1, x_2}, D_{x_1, x_2}>$,
\item $P$~is of degree $O(d_f^2 r_f^2)$ in~$x_1x_2$,
of total degree $O(d_f^9r_f^8)$ in $x_3, \ldots, x_n$, and
of total degree $O(d_f^2 r_f^2)$ in $T_{x_1, x_2}, D_{x_1, x_2}$,
\end{itemize}
and for each $h \in \{ 3,\dots, n \}$, there exists a non-zero annihilating operator~$Q_h$ of~$f$ that satisfies
\begin{itemize}
\item $Q_h\in K[x_3, \dots, x_n][x_1 x_2]\<T_{x_1, x_2}, D_{x_h}>$,
\item $Q_h$~is of degree $O(d_f^2 r_f^3)$ in~$x_1x_2$,
of total degree $O(d_f^9r_f^{12})$ in $x_3, \ldots, x_n$, and of total degree $O(d_f^2 r_f^3)$ in $T_{x_1, x_2}, D_{x_h}$.
\end{itemize}
\end{theorem}

\begin{proof}
First we prove the existence of the operator~$P$.
We apply the counting argument used in~\cite{Gessel1981, Lipshitz1988}.
Use Definition~\ref{DEF:C} with~$S = \{1,2\}$.
For any positive integer~$N$, set
\[ V_N = \text{span}_{K(x_3,\dots,x_n)}\left\{ C^{2 N}  (x_1 x_2)^i T_{x_1, x_2}^j D_{x_1, x_2}^\ell \, \mid  \, i+j+\ell \le N\right\} \]
and
\begin{equation*}
W_N = \text{span}_{K(x_3,\dots,x_n)}\mycH_{2N(d_1+d_2+1), 0} .
\end{equation*}
By degree considerations, for any integers $i,j,\ell$ satisfying $i+j+\ell\leq N$ we have
\begin{equation*}
(x_1 x_2)^i T_{x_1, x_2}^j D_{x_1, x_2}^\ell \in \cF_{j+2i, \, j+2\ell} \subseteq \mycH_{j+2i, \, j+2\ell} \subseteq \mycH_{2N, \, 2N} .
\end{equation*}
Note that $\tdeg(C) \le d_C = d_1+d_2$. Hence by Lemma~\ref{LM:reduce-2},
\begin{equation}\label{eq:ijl-reduce}
C^{2 N}  (x_1 x_2)^i T_{x_1, x_2}^j D_{x_1, x_2}^\ell \in \mycH_{2N(d_1+d_2+1), \,0} + J .
\end{equation}
Consequently, we have the inclusion
$V_N \subseteq W_N + K(x_3,\dots,x_n) \,J$ between $K(x_3,\dots,x_n)$-vector spaces.
Note the asymptotic estimates
\begin{equation*}
\dim_{K(x_3,\dots,x_n)}V_N = \binom{N+3}{3} = \Theta(N^{3}) ,
\end{equation*}
where the first equality is by Lemma~\ref{LM:basis}, and
\begin{equation*}
\dim_{K(x_3,\dots,x_n)} W_N = r_1 r_2 \binom{2N(d_1+d_2+1)+2}{2} = \Theta(N^2) .
\end{equation*}
Choosing sufficient large~$N$ results in $\dim(V_N) > \dim(W_N)$.
So, some non-zero element of $V_N$ is in $K(x_3,\dots,x_n)J $
and without loss of generality we can choose it in~$\bW_n \cap V_N$.
Observe that this operator has $C^{2N}$ as a left factor.
So, dividing by~$C^{2N}$ yields a non-zero annihilating operator of~$f$
in~$K[x_1 x_2, x_3, \dots, x_n] \< D_{x_1}, D_{x_2}>$.

To control the degree and order of such an annihilating operator,
we now make a more specific choice that will lead to the announced operator~$P$.
To this end,
we make \eqref{eq:ijl-reduce} explicit in the form
\begin{equation*}
C^{2N}(x_1x_2)^i T_{x_1, x_2}^j D_{x_1, x_2}^\ell
\in \ \ \sum_{\mathclap{\substack{i_1 < r_1 , \ i_2< r_2 , \\ k_1+k_2 \leq 2N(d_1 + d_2 + 1)}}}
q_{i, j, \ell, i_1, i_2, k_1, k_2} x_1^{k_1}x_2^{k_2}D_{x_1}^{i_1} D_{x_2}^{i_2} + J ,
\end{equation*}
for polynomials $q_{i, j, \ell, i_1, i_2, k_1, k_2}$ of $K[x_3, \ldots, x_n]$ of total degree bounded by $2N(d_1+d_2+1)$,
and we set up an ansatz of the form
\begin{equation}\label{eq:ansatz}
\begin{aligned}
C^{2N} P &= \sum_{i+j+\ell \leq N} p_{i,j,\ell} C^{2N}(x_1x_2)^i T_{x_1, x_2}^j D_{x_1, x_2}^\ell \\
&\in \ \  \sum_{\mathclap{\substack{i_1 < r_1 , \ i_2< r_2 , \\ k_1+k_2 \leq 2N(d_1 + d_2 + 1)}}}
q_{i_1, i_2, k_1, k_2} x_1^{k_1}x_2^{k_2}D_{x_1}^{i_1} D_{x_2}^{i_2} + J ,
\end{aligned}
\end{equation}
where the~$p_{i,j,\ell}$ are undetermined polynomials from~$K[x_3,\dots,x_n]$
and the resulting coefficients~$q_{i_1, i_2, k_1, k_2}$ are polynomials of $K[x_3, \ldots, x_n]$
given as linear combinations of the~$p_{i,j,\ell}$ by
\begin{equation*}
q_{i_1, i_2, k_1, k_2} = \sum_{i+j+\ell \leq N} p_{i,j,\ell} \, q_{i, j, \ell, i_1, i_2, k_1, k_2} .
\end{equation*}
After applying to~$f$ to obtain
\begin{equation*}
C^{2N} P(f) =
\ \ \sum_{\mathclap{\substack{i_1 < r_1 , \ i_2< r_2 , \\ k_1+k_2 \leq 2N(d_1 + d_2 + 1)}}}
q_{i_1, i_2, k_1, k_2} x_1^{k_1}x_2^{k_2}D_{x_1}^{i_1} D_{x_2}^{i_2}(f) ,
\end{equation*}
we can enforce $P(f) = 0$ by forcing each~$q_{i_1, i_2, k_1, k_2}$ to be zero.
This gives a linear system over $K(x_3, \ldots, x_n)$ with $\binom{N+3}{3}$ variables and a number~$S$ of equations that is
\[S := \dim_{K(x_3,\dots,x_n)} W_N = r_1 r_2 \binom{2N(d_1+d_2+1)+2}{2}.\]
Set $R:=r_1r_2$ and $D:= 2(d_1+d_2+1) \geq 2$.
By Observation~\ref{LM:ineq1}, we can choose $N := 3 D^2 R$
so as to get a system with more variables than equations and thus a system with a non-trivial solution.
Because the corresponding polynomial matrix is of size $S \times \binom{N+3}{3}$ with entries of total degree $2N(d_1+d_2+1)$,
by Lemma~\ref{LM:deg} we have, for a suitable non-zero solution~$(p_{i, j, \ell})$,
\[\text{tdeg}(p_{i, j, \ell}) \leq 2N(d_1+d_2+1)r_1r_2 \binom{2N(d_1+d_2+1)+2}{2} = O(d_f^9r_f^8), \]
where the total degree is with respect to $x_3, \ldots, x_n$.
This non-trivial solution leads to a non-zero annihilator $P\in K[x_1 x_2, x_3, \dots, x_n]\<T_{x_1, x_2}, D_{x_1, x_2}>$ of~$f$.
From the ansatz form~\eqref{eq:ansatz}, $P$~has
its degree in~$x_1x_2$ bounded by $N = O(d_f^2r_f^2)$
and its total degree in $T_{x_1, x_2}, D_{x_1, x_2}$ not exceeding $N = O(d_f^2 r_f^2)$.
This leads to the desired degree and order bounds for~$P$.

For each $ h \in \{3,\dots, n\} $, the proof of the existence of the operator $ Q_h $ is similar.
Using Definition~\ref{DEF:C} with $S = \{1,2,h\}$, we set
\[ V_N = \text{span}_{K(x_3,\dots,x_n)}\left\{ C^{N}  (x_1 x_2)^i T_{x_1, x_2}^j D_{x_h}^\ell \, \mid  \, i+j+\ell \le N\right\}, \]
and
\[ W_N = \text{span}_{K(x_3,\dots,x_n)} \mycH_{N(d_1+d_2+d_h+2), \,0}. \]
This time we derive $V_N \subseteq \mycH_{2N, \, N}$ (not $\mycH_{2N, \, 2N}$)
and we have the additional term~$d_h$ in $\tdeg(C) \le d_C = d_1+d_2+d_h$,
so that the analogue of~\eqref{eq:ijl-reduce} is
\begin{equation*}
C^N  (x_1 x_2)^i T_{x_1, x_2}^j D_{x_h}^\ell \in \mycH_{N(d_1+d_2+d_h+2), \,0} + J .
\end{equation*}
Set $ R:=r_1r_2r_h $ and $ D:= d_1+d_2+d_h+2 \geq 2 $.
We can still choose
\[ N := 3 D^2 R = 3 \,(d_1+d_2+d_h+2)^2 r_1 r_2 r_h. \]
Then by Observation~\ref{LM:ineq1}
\[\dim_{K(x_3,\dots,x_n)} V_N - \dim_{K(x_3,\dots,x_n)} W_N = \binom{N+3}{3} - R \binom{DN+2}{2} > 0. \]
Continuing as we did for~$P$, we obtain that there exists a non-zero operator
\[ Q_h\in K[x_1 x_2, x_3, \dots, x_n]\<T_{x_1, x_2}, D_{x_h}> \]
such that $ Q_h(f)=0 $.
By a similar argument, we have that $Q_h$~is of degree at most $N = O(d_f^2 r_f^3)$ in~$x_1x_2$, of total degree $O(d_f^9r_f^{12})$ in~$x_3, \ldots, x_n$, and of total degree at most $N = O(d_f^2 r_f^3)$ in $T_{x_1, x_2}, D_{x_h}$.
\end{proof}

After the preparation above, let us prove the diagonal theorem.

\begin{proof}[Proof of Theorem~\ref{THM:diagonal-2}]
Let $u_1, \dots, u_n$ be new variables.
Write $K\<\<u_1, \dots, u_n>>$ for the associative $K$-algebra over the free non-commutative monoid generated by $\lbrace u_1, \dots, u_n\rbrace$.
Assume that $f\in K[[\vx]]$ is D-finite over $K(\vx)$.
By Theorem~\ref{THM:elim}, there exists a non-zero operator~$P$ in~$K[x_1 x_2, x_3, \dots, x_n]\<T_{x_1, x_2}, D_{x_1, x_2}>$
and, for each $h \in \{ 3,\dots, n \}$, a non-zero operator~$Q_h$ in~$K[x_1 x_2, x_3, \dots, x_n]\<T_{x_1, x_2}, D_{x_h}>$
such that $P(f)=0$ and for each $ h \in \{ 3,\dots, n \} $, $ Q_h(f)=0 $.

We first show that there is a non-zero operator $\bar{P} \in K(x_1, x_3, \dots, x_n)\<D_{x_1}>$ such that $  \bar{P} (\Delta_{1,2} (f))=0$.
Recall that $ T_{x_1, x_2} $ commutes with $ x_1x_2 $ and $ D_{x_1, x_2} $. Consider the maximal integer $ s $ such that
\begin{equation}
P = T_{x_1, x_2}^{s} \tilde{P} \quad  \text{with} \quad \tilde{P}=\sum_{i=0}^{m} T_{x_1, x_2}^i A_i(x_1 x_2, x_3, \dots, x_n, D_{x_1, x_2})
\end{equation}
for some~$A_i \in K\<\<u_1, \dots, u_n>>$, where $ A_i (\sigma_1, \dots, \sigma_n) $ denotes the evaluation at $ u_1 = \sigma_1, \dots, u_n = \sigma_n $ of~$A_i$ for elements $ \sigma_1, \dots, \sigma_n \in \bW_n $.
	The maximality of $ s $ implies $ A_0 \neq 0 $. By Lemma~\ref{LEM:xy}, we have
	\begin{equation} \label{eq3}
		\tilde{P}(f)=\sum_{i=0}^{m} T_{x_1, x_2}^i A_i(f) = g(x_1 x_2, x_3, \dots, x_n)
	\end{equation}
	for some power series $g$ in $ n-1 $ variables.
	Since $ \Delta_{1,2}T_{x_1, x_2}=0 $ and by Proposition \ref{PROP:diagonal}, taking the diagonal of the two sides of \eqref{eq3} yields
	$$  \Delta_{1,2}\tilde{P}(f)=A_0(x_1, x_3, \dots, x_n, D_{x_1}\theta_{x_1})(\Delta_{1,2}(f)) = g(x_1,x_3,\dots,x_n).$$
	The operator $ H := A_0(x_1, x_3, \dots, x_n, D_{x_1}\theta_{x_1}) $ is non-zero, since $$ x_1, x_3, \dots, x_n, D_{x_1}\theta_{x_1} $$ are linearly independent over $K$ by Lemma \ref{LM:basis-2}.
	Because $ f $ is D-finite over $ K(\vx) $, the series $ \tilde{P}(f) $ is also D-finite over $ K(\vx) $. Hence $ g(x_1,x_3,\dots,x_n) = \tilde{P}(f) |_{x_2=1} $ is D-finite over $ K(x_1, x_3, \dots, x_n) $ by Lemma \ref{LEM:dfinite-alg}.
	Therefore there exists a non-zero operator $G \in K(x_1, x_3, \dots, x_n)\<D_{x_1}>$ such that $G(g)=0$.
	Then the operator $ \bar{P} := G H $ is non-zero and $  \bar{P} (\Delta_{1,2} (f))=0$.

	The existence of a non-zero operator $ \bar{Q}_h \in K(x_1, x_3, \dots, x_n)\<D_{x_h}> $ such that $\bar{Q}_h(\Delta_{1,2} (f))=0$ for each $ h \in \{ 3,\dots, n \} $ is proved similarly. The only difference is the variation in the formula
	$$
	\begin{aligned}
		\Delta_{1,2}{A}_0(x_1x_2, x_3, \dots, x_n, D_{x_h}) (f ) &=  A_0(x_1, x_3, \dots, x_n, D_{x_h}) \Delta_{1,2} (f) \\ &=  g(x_1,x_3,\dots,x_n).
	\end{aligned}
    $$
    Hence we conclude that $\Delta_{1,2}(f)$ is D-finite over $K(x_1, x_3, \dots, x_n)$.
\end{proof}

The following result is very much inspired by~\cite{Kauers-2014-BDC},
which we merely generalize to the bivariate situation.
The reader will pay attention that it combines
bounds about a function~$f$ provided by a system of equations,
each in a single derivative like in Definition~\ref{DEF:bound},
with bounds on a (potentially) partial differential operator~$L$,
to derive bounds on equations in a single derivative for~$L(f)$.

\begin{lemma}\label{lem:bound-for-L(f)}
Fix $n=2$ and a bivariate D-finite function~$f$.
Given a system of linear differential equations with known order and degree bounds $r_f$ and~$d_f$
exhibiting the D-finiteness of~$f$,
as well as an operator~$L$ of order~$r_L$ and degree~$d_L$,
there exists a system of linear differential equations
exhibiting the D-finiteness of~$g = L(f)$,
whose order~$r_g$ and degree~$d_g$ are bounded by
\begin{equation}\label{eq:bound-for-L(f)}
d_g \leq (d_L+2d_f(r_f^2+r_L)) r_f^2
\quad\text{and}\quad
r_g \leq r_f^2 .
\end{equation}
\end{lemma}

\begin{proof}
Use Definition~\ref{DEF:C} when~$S = \{1,2\}$.
We look for non-zero operators~$A \in K[x_1 x_2] \< D_{x_1}>$ annihilating~$g$,
that is, such that $(AL)(f) = 0$.
Write $r_A$ and~$d_A$ for the order and degree of a potential~$A$.
For $l \in K[x_1,x_2]$,
if $\deg(l) \leq d_L$, $0\leq k\leq r_A$, and $0\leq i+j\leq r_L$, then, by Lemma~\ref{LM:reduce-2} we have
\begin{equation*}
C^{r_A+r_L} D_{x_1}^k l(x_1,x_2) (D_{x_1}^iD_{x_2}^j) \in \mycH_{d_L+d_C(r_A+r_L),0} + J ,
\end{equation*}
hence for a potential $A = \sum_{k=0}^{r_A} a_k(x_1,x_2) D_{x_1}^k$ we need to have
\begin{equation*}
C^{r_A+r_L} (AL)(f) = \sum_{0\leq i<r_1, \, 0\leq j<r_2} \sum_{k=0}^{r_A} a_k q_{i,j,k} D_{x_1}^iD_{x_2}^j(f)
\end{equation*}
for explicit polynomials~$q_{i,j,k} \in K[x_1,x_2]$ of degree at most $d_L+d_C(r_A+r_L)$.
Now, for this to be zero, the $r_A+1$ polynomial coefficients of~$A$
need to cancel the $r_1r_2 = O(r_f^2)$ equations obtained
by equating the coefficients of the $K[x_1,x_2]$-linearly independent elements $D_{x_1}^iD_{x_2}^j(f)$ that appear in the sum.
Setting $r_A = r_1r_2$ ensures a non-zero solution exist,
and Lemma~\ref{LM:deg}~guarantees there exists a solution
with degree~$d_A$ at most $(d_L+d_C(r_A+r_L)) r_A$.
Looking for~$A \in K[x_1 x_2] \< D_{x_2}>$ leads to the same bounds,
which leads to~\eqref{eq:bound-for-L(f)}.

\end{proof}

\begin{cor}\label{cor:bivariate-gessel}
Let $f\in K[[x_1,x_2]]$ be D-finite over $K(x_1,x_2)$.
Then $\Delta_{1,2}(f)$ is D-finite over $K(x_1)$.
In addition, there exists a non-zero operator~$\bar P$ that satisfies
$\bar P(\Delta_{1,2}(f)) = 0$ and
\begin{equation*}
{\deg(\bar P) = O(d_f^3 r_f^{4})}
\quad\text{and}\quad
{\ord(\bar P) = O(d_f^2 r_f^2)} .
\end{equation*}
\end{cor}

\begin{proof}
The first statement is just Theorem~\ref{THM:diagonal-2} in the case~$n=2$.
For the degree bounds, we continue in the context of the proof of Theorem~\ref{THM:diagonal-2}.
Specifically, we have found:
\begin{itemize}
\item an operator $\tilde P = \tilde P(x_1x_2,T_{x_1,x_2},D_{x_1,x_2})$ that is a factor of an operator~$P$
that we obtained by Theorem~\ref{THM:elim} and therefore satisfies
that its degree in~$x_1x_2$ and its degree in~$D_{x_1, x_2}$ are both~$O(d_f^2 r_f^2)$,
\item a univariate power series~$g$ such that $\tilde P(f) = g(x_1x_2)$,
\item a non-zero operator~$H = H(x_1,D_{x_1}\theta_{x_1})$ such that
  $H(x_1x_2,D_{x_1,x_2})$~is the coefficient of~$T_{x_1,x_2}^0$ in~$\tilde P$
  and $H(\Delta_{1,2}(f)) = g(x_1)$.
\end{itemize}
By construction, both $\tilde P$ and~$H$ admit the same bounds on order and degree as~$P$,
in particular, both $\ord(H)$ and~$\deg(H)$ are in~$O(d_f^2 r_f^2)$.
Now, Lemma~\ref{lem:bound-for-L(f)} applies to the D-finite function~$f$ and the operator~$H$
to prove the existence of a non-zero annihilator $G \in K[x_1]\<D_{x_1}>$ of~$g$ satisfying
\begin{equation*}
\deg(G) \leq (\deg(H)+2d_f(r_f^2+\ord(H))) r_f^2 = O(d_f^3 r_f^4)
\quad\text{and}\quad
\ord(G) \leq r_f^2
\end{equation*}
as a consequence of~\eqref{eq:bound-for-L(f)}.
Setting $\bar L = GH$ and observing that $H$~has lower bounds than~$g$
gives the announced result.
\end{proof}

\begin{remark}
It is unsatisfactory
that we could not find and apply a one-stage variant of Gessel's approach,
especially in view of the bivariate case
in which it outperforms Lipshitz's approach that is developed in the next section.
After this work, it would still be of interest to derive such a direct variant.
\end{remark}

\section{Lipshitz's method for bounds of diagonal}\label{SEC:bound}

In this section,
we analyze the method of Lipshitz \cite{Lipshitz1988}
and we make specific choices in it
so as to construct annihilating operators of a diagonal
and to derive upper bounds on their order and degree.

Let us provide definitions that generalize those of Section~\ref{SECT:intro}.
Given integers $n$ and~$m$ satisfying $0 \leq m \leq n-1$,
we use the notation $\vs$ for $s_1, \dots, s_m$ and $\hx$ for $x_{m+1}, \dots , x_n$.
In particular, the list~$\vs$ is empty if~$m=0$,
which was the setting in Section~\ref{SECT:intro}.
The variable~$x_{m+1}$ is denoted by~$t$ if $m \geq 1$:
in this new situation,
our goal is to take a diagonal with respect to $\vs,t = s_1,\dots,s_m,x_{m+1}$,
keeping $\hx = x_{m+2},\dots,x_n$ as parameters.
For primary diagonals there is a single~$s_i$ ($m=1$),
and we simply denote~$s_1$ by~$s$.
In other words, we have:
\begin{equation*}
\vs, \hx = \begin{cases}
x_1, \dots, x_n & \text{if} \quad m = 0 , \\
s, t \, (= x_2), x_3, \dots, x_n & \text{if} \quad m = 1 , \\
s_1, \dots, s_m, t \, (= x_{m+1}), x_{m+2}, \dots, x_n & \text{if} \quad m \geq 2 .
\end{cases}
\end{equation*}

The definitions of~$\tau$ that will be needed,
\eqref{EQ:simple-tau}~in the present section
and \eqref{EQ:gen-tau} in Section~\ref{SECT:single-step},
motivate that we accommodate series with negative exponents
by defining
\begin{equation*}
M := \bigcup_{k \in \bN} \bigoplus_{|\aalpha| + |\bbeta| \geq -k} K \vs^\aalpha \hx^\bbeta \subseteq K^{\bZ^m \times \bN^{n-m}} ,
\end{equation*}
where $\aalpha := (\alpha_1,\dots,\alpha_m) \in \bZ^m$
and $\bbeta := (\beta_{m+1},\dots,\beta_n) \in \bN^{n-m}$.
This set~$M$ is a module over $K[\vs, \hx] \< \vD_s, \vD_{\hx} >$,
but it is not a $K(\vs, \hx)$-vector space.
If $m = 0$,
then $\vx = \hx$ and $M$~is just the ring $K[[\vx]]$ of formal power series.

\begin{definition}[D-finiteness]\label{DEF:dfinite2}
An element $F\in M$ is \emph{D-finite} over $K(\vs, \hx)$
if the $K(\vs, \hx)$-vector space generated by the derivatives of~$F$
in $\cT := K(\vs, \hx) \otimes_{K[\vs, \hx]}M $ is finite-dimensional,
after identifying each element $m \in M$ with $1\otimes m \in \cT$.
\end{definition}

The reader will pay attention to the redefinition of a number of quantities in Sections \ref{SECT:Ds-Dt} and~\ref{SECT:Ds-Dx},
including $M$, $S$, $B$, $C$, $d_C$, $R$, $N$, $\cG_N$, $V_N$, $W_N$, $\phi$.

\subsection{Bounds for primary diagonal}

We analyze the behavior of the primary diagonal operator~$\Delta_{2,1}$ and derive the following theorem,
which gives bounds on order and degree for linear differential operators that annihilate $\Delta_{2,1}(f)$.
The rest of the section consists of the proof of this theorem,
with the bounds~\eqref{EQ:bounds-on-P-alpha} proven by Lemma~\ref{LEM:exists-P-alpha}
and the bounds~\eqref{EQ:bounds-on-P-h-alpha} proven by Lemma~\ref{LEM:exists-P-h-alpha}.

\begin{theorem}\label{THM:primary bound}
Let $f\in K[[\vx]]$ be {D-finite} over $K(\vx)$ and let $ d_i, f_i, d_f, r_f $ be as in Definition~\ref{DEF:bound}.
Then, there exists a non-zero annihilating operator~$P_\alpha$ of~$\Delta_{2,1}(f)$
in $K[t, x_3, \dots, x_n]\<D_t>$ that satisfies
\begin{equation}\label{EQ:bounds-on-P-alpha}
\begin{aligned}
\deg(P_\alpha) &\leq 8 (d_1+d_2+1)^2 (r_1 r_2)^2 (8(d_1+d_2+1)^2r_1r_2+1) = O(d_f^{4} r_f^6) , \\
\ord(P_\alpha) &\leq 4(d_1+d_2+1)\,r_1 r_2= O(d_f r_f^2) ,
\end{aligned}
\end{equation}
and for each $h \in \{3,\dots,n\}$, there exists a non-zero annihilating operator~$P_{h,\alpha_h}$ of~$\Delta_{2,1}(f)$
in $K[t, x_3, \dots, x_n]\<D_{x_h}>$ that satisfies
\begin{equation}\label{EQ:bounds-on-P-h-alpha}
\begin{aligned}
\deg(P_{h, \alpha_h}) &\leq 8 (d_1+d_2 +d_h +1)^2 (r_1 r_2 \, r_h)^2 (8(d_1+d_2+d_h+1)^2r_1r_2 \, r_h+1) \\
  &\qquad = O(d_f^{4} r_f^9), \\
\ord(P_{h, \alpha_h}) &\leq 4(d_1+d_2+d_h+1)\,r_1 r_2 \, r_h= O(d_f r_f^3).
\end{aligned}
\end{equation}
\end{theorem}

We specialize our setting by choosing~$m=1$,
that is, we make $\vs, \hx = s, t, x_3, \dots, x_n$.
We aim to refine Lipshitz's proof \cite[Lemma~3]{Lipshitz1988}
of existence of annihilating operators in $K[\hx]\<D_s,D_{x_i}>$
for $i=2,\dots,n$.
Recall the notation $\cS = K(\vx) \otimes_{K[\vx]} K[[\vx]]$
from the introduction.
We define two maps $\sigma$ and~$\tau$ from~$\cS$ to~$M$ by
\begin{equation}\label{EQ:simple-tau}
\tau(h(\vx)) = h\left(s, \frac{t}{s}, x_3, \dots, x_n\right)
\quad\text{and}\quad
\sigma(h(\vx)) = \frac{\tau(h(\vx))}{s} .
\end{equation}
Hence, $\tau$ is a ring morphism and we have
\begin{equation}\label{EQ:non-morph}
\sigma (gh) = \tau(g) \sigma(h)
\quad\text{for any $g,h$ in~$\cS$} .
\end{equation}

\begin{lemma}\label{LEM:coeff-extraction-univ}
Let $P$ be any non-zero operator
\begin{equation}\label{EQ:def-P-univ}
P = P(\hx; D_t, D_s) = \sum_{j = \alpha}^{\beta} P_j(\hx; D_t) D_s^j
\in K[\hx]\<D_t, D_s>
\end{equation}
for which~$P_\alpha \neq 0$,
and let $g \in \sum_{i \in \bZ} g_i(\hx) s^i$ be any element of~$M$.
Then, the coefficient of~$s^{-1-\alpha}$ in~$P(g)$
is~$P_\alpha(g_{-1})$.

Similarly, for any $h \in \{3, \dots, n\}$, if $P_h$~is a non-zero operator
\begin{equation}\label{EQ:def-P-h-univ}
P_h = P_h(\hx; D_{x_h}, D_s) = \sum_{j = \alpha_h}^{ \beta_h} P_{h,j}(\hx; D_{x_h}) D_s^j
\in K[\hx]\<D_{x_h}, D_s> ,
\end{equation}
for which~$P_{h,\alpha_h} \neq 0$,
then the coefficient of~$s^{-1-\alpha}$ in~$P_h(g)$
is~$P_{\alpha_h}(g_{-1})$.
\end{lemma}

\begin{proof}
Note that
\begin{equation*}
D_s^j(g) = D_s^j\biggl(\sum_{i\leq-2} g_i(\hx)s^i\biggr)
         + (-1)^jj! \, g_{-1}(\hx)s^{-1-j}
         + D_s^j\biggl(\sum_{i\geq0} g_i(\hx)s^i\biggr) ,
\end{equation*}
where the first term has all exponents less than~$-1-j$
and the last has all exponents at least~$0$:
only the middle term contributes to the coefficient of~$s^{-1-j}$.
So, for~$j \geq \alpha$, some contribution to the coefficient of~$s^{-1-\alpha}$
is only possible if~$j = \alpha$,
proving the result for the case $P = P(\hx; D_t, D_s)$.
The proof for the other cases is the same.
\end{proof}

Consider any non-necessarily D-finite series
\begin{equation}\label{EQ:def-f}
f = \sum_{i_1,\ldots, i_n \geq 0} a_{i_1,\ldots, i_n} x_1^{i_1}\cdots x_n^{i_n} \in K[[\vx]]
\end{equation}
and the corresponding element~$\sigma(f)$ of~$M \subseteq \cT$.
By Definition~\ref{DEF:diag} (diagonals) and because we write~$t$ for~$x_2$,
the primary diagonal~$\Delta_{2,1}(f)$ is
\begin{equation*}
\Delta_{2,1}(f) = \sum_{i_1,i_3, \ldots, i_n \geq 0} a_{i_1,i_1,i_3,\ldots, i_n} t^{i_1} x_3^{i_3}\cdots x_n^{i_n} \in K[[\hx]] .
\end{equation*}
By the definition~\eqref{EQ:simple-tau} of $\tau$ and~$\sigma$,
this diagonal is the coefficient of degree~$s^{-1}$ in~$\sigma(f)$.
The following lemma immediately follows, as a consequence of Lemma~\ref{LEM:coeff-extraction-univ}.

\begin{lemma}\label{LEM:ann-univ}
Let $f$ be as in~\eqref{EQ:def-f}.
If $P(f) = 0$ for $P$ and~$P_\alpha\neq0$ as in~\eqref{EQ:def-P-univ},
then $P_\alpha$~annihilates~$\Delta_{2,1}(f)$.
For any~$h \in \{3, \dots, n\}$,
if $P_h(f) = 0$ for $P_h$ and~$P_{h,\alpha_h}\neq0$ as in~\eqref{EQ:def-P-h-univ},
then $P_{h,\alpha_h}$~annihilates~$\Delta_{2,1}(f)$.
\end{lemma}

In the next two subsections, when $f$~is D-finite
we will construct operators $P$ and~$P_h$ to be used in the previous lemma.

\subsubsection{Controlling and combining the $D_s^iD_t^j(\sigma(f))$}
\label{SECT:Ds-Dt}

We construct an operator $P \in K[\hx]\<D_t, D_s>$ such that $P(\sigma(f)) = 0$.
To this end, we introduce two vector spaces depending on~$N \in \bN$,
\begin{equation}\label{EQ:def-V-N}
V_N = A_N(s,\hx) \ \span_{K(\hx)} \{   D_s^i D_t^j  \; | \;\, i+j \leq N\}
\end{equation}
and
\begin{equation}\label{EQ:def-W-N}
W_N = \span_{K(\hx)} \{ s^\alpha \sigma(\Dx^\bbeta f) \; | \;\,  \alpha \leq DN, \bbeta \in B \} ,
\end{equation}
where $B$~is a finite set and $A_N(s,\hx)$~is a polynomial,
both to be determined (see Lemma~\ref{LEM:D-sigma-in-vspace}).
We will prove that the map defined by $\phi (P) := P (\sigma(f))$
is $K(\hx)$-linear from~$V_N$ to~$W_N$,
that it is non-injective for large enough~$N$
(see Lemma~\ref{LEM:exists-P-alpha}).
As a by-product,
we will get an annihilator~$P_\alpha$ of $\Delta_{2,1}(f)$
with controled degree and order
(see again Lemma~\ref{LEM:exists-P-alpha}).

Denote $D_i := D_{x_i}$ for $i = 1,\dots, n$.

\begin{lemma}\label{LEM:D-sigma-once}
We have for all $g \in \cS$:
\begin{equation*}
\begin{aligned}
D_{s}(\sigma(g)) &= \sigma \left( \left(-x_1^{-1} + D_1 - x_1^{-1}x_2 D_2 \right) (g) \right), \\
D_t(\sigma(g)) &= \sigma\left((x_1^{-1}D_2)(g)\right), \\
D_{x_{h}}(\sigma(g)) &= \sigma(D_{h}(g)), \quad h = 3,\dots, n.
\end{aligned}
\end{equation*}
\end{lemma}

\begin{proof}
For the first two identities, write the following two equations
by the chain rule, then use
the formulas $\tau(x_1) = s$, $\tau(x_2) = t/s$, and~\eqref{EQ:non-morph}:
\begin{equation*}
\begin{aligned}
D_{s}(\sigma(g))
  &= -\frac{1}{s} \sigma(g) + \sigma(D_1(g)) - \frac{t}{s^2}\sigma(D_2 (g)) , \\
D_t(\sigma(g))
  &= \frac{1}{s} \sigma(D_{2}(g)) . \\
\end{aligned}
\end{equation*}
The third identity is obvious.
\end{proof}

Define for any~$N \in \bN$:
\begin{equation}\label{EQ:G}
\cG_N := \bigoplus_{a+b \leq N} x_1^{-N} K[x_1, x_2]_{\leq N} D_1^a D_2^b .
\end{equation}

\begin{lemma}\label{LEM:D-sigma-in-G}
For all $g \in \cS$ and all non-negative integers $i$ and~$j$,
$D_t^j D_s^i (\sigma(g))$~is an element of~$\sigma(\cG_{i+j}(g))$.
\end{lemma}

\begin{proof}
It follows immediately from Lemma~\ref{LEM:D-sigma-once} that, for all $i,j \in \bN$,
\begin{equation}\label{EQ:diff-sigma}
\begin{aligned}
D_t^j D_s^i (\sigma(g)) &= D_t^j \sigma \left( (-x_1^{-1} + D_1 - x_1^{-1}x_2 D_2)^i (g) \right) \\
&= \sigma\left((x_1^{-1}D_2)^j (-x_1^{-1} + D_1 - x_1^{-1}x_2 D_2)^i(g)\right) .
\end{aligned}
\end{equation}
Consider an element~$x^{-i}p D_1^aD_2^b$ of~$\cG_i$,
or equivalently, integers $a$ and~$b$ and a polynomial $p \in K[x_1,x_2]$
satisfying $a+b \leq i$ and~$\tdeg(p) \leq i$.
We observe that
\begin{multline*}
\left(-\frac{1}{x_1} + D_1 - \frac{x_2}{x_1} D_2 \right) \left(\frac{p}{x_1^i} D_1^a D_2^b \right) = \\
\frac{1}{x_1^{i+1}} \biggl( -p D_1^a D_2^b - ip D_1^a D_2^b + x_1 D_1(p) D_1^a D_2^b + x_1 p D_1^{a+1} D_2^b \\
- x_2 D_2(p) D_1^a D_2^b - x_2 p D_1^a D_2^{b+1} \biggr)
\end{multline*}
is in~$\cG_{i+1}$.
Therefore,
$\left(-\frac{1}{x_1} + D_1 - \frac{x_2}{x_1} D_2 \right) \cG_i \subseteq \cG_{i+1}$,
by linearity.
We derive similarly
\begin{equation*}
\left(\frac{1}{x_1}D_2\right)
\left(\frac{1}{x_1^i} p(x_1,x_2) D_1^a D_2^b \right) =
\frac{1}{x_1^{i+1}} \left( D_2(p) D_1^a D_2^b + p D_1^a D_2^{b+1} \right) \in \cG_{i+1} ,
\end{equation*}
and $\left(\frac{1}{x_1}D_2\right) \cG_i \subseteq \cG_{i+1}$.
Since $1 \in \cG_0$, we get by induction that for all $i,j \in \bN$,
\begin{equation*}
\left(\frac{1}{x_1}D_2\right)^j
\left(-\frac{1}{x_1} + D_1 - \frac{x_2}{x_1} D_2 \right)^i \in \cG_{i+j} .
\end{equation*}
\end{proof}

\begin{lemma}\label{LM:tau-vspace}
For any integers $p$ and~$q$, we have:
\begin{equation*}
\tau\left(\frac{1}{x_1^q} K[x_1,x_2]_{\leq p}\right) \subseteq
  \frac{K[s,t]_{\leq 2p}}{s^{p+q}}
\quad\text{and}\quad
\tau\left(K[\vx]_{\leq p}\right) \subseteq
  \frac{K[s,\hx]_{\leq 2p}}{s^p} .
\end{equation*}
\end{lemma}

\begin{proof}
Both formulas follow by linearity from the action of~$\tau$ on monomials:
\begin{equation*}
\begin{aligned}
\tau(x_1^i x_2^j) &= \frac{s^{i+(p-j)}t^j}{s^p} \in
  \frac{K[s,t]_{\leq p+i}}{s^p} \quad\text{if}\quad i+j\leq p ; \\
\tau(\vx^\vi) &= \frac{s^{i_1+(p-i_2)} x_2^{i_2}\dotsm x_n^{i_n}}{s^p} \in
  \frac{K[s,\hx]_{\leq p+i_1}}{s^p} \quad\text{if}\quad |\vi|\leq p . \\
\end{aligned}
\end{equation*}
\end{proof}

\begin{lemma}\label{LEM:D-sigma-in-vspace}
Consider
$B := \lbrace 0, 1, \dots, r_1-1\rbrace \times \lbrace 0, 1, \dots, r_2-1\rbrace$,
the polynomial~$C$,
and $d_C = d_1+d_2\leq 2 d_f$
as set by Definition~\ref{DEF:C} for $S := \{1, 2\}$.
Fix~$N \in \bN$ and set $D := 2+2 d_C \geq 2$ and
$A_N (s,\hx) := s^{(d_C+2) N } \tau(C^N)\in K[s,\hx]$.
Then, if $i+j \leq N$, then
\begin{equation}\label{EQ:D-of-sigma}
D_s^i D_t^j (\sigma(f)) \in \sum_{\alpha \leq DN \atop \bbeta \in B} \frac{K[\hx]_{\leq D N}}{A_N(s,\hx)} s^\alpha \sigma(\vD^\bbeta_x f) .
\end{equation}
\end{lemma}

\begin{proof}
If $i+j \leq N$, then Lemma~\ref{LEM:D-sigma-in-G}, Equation~\eqref{EQ:non-morph} and Lemma~\ref{LM:tau-vspace} imply
\begin{equation}\label{EQ:move-D-inside-sigma}
\begin{aligned}
D_s^i D_t^j (\sigma(f)) \in \sigma (\cG_{i+j}(f)) &\subseteq \sigma (\cG_N(f)) \\
  &= \sum_{a+b \leq N} \tau(x_1^{-N} K[x_1,x_2]_{\leq N}) \, \sigma(D_1^a D_2^b (f)) \\
  &\subseteq \sum_{a+b \leq N} \frac{K[s,t]_{\leq 2N}}{s^{2N}} \, \sigma(D_1^a D_2^b(f)) .
\end{aligned}
\end{equation}
Next,
by Definition~\ref{DEF:C} for $S := \{1, 2\}$
and by Lemma~\ref{LM:reduce-2} with $u := N \geq v:= a+b$ and~$t := 0$, we have
\[
D_1^a D_2^b \in \mycH_{0, a+b} \subseteq \frac{1}{C^N} \, \mycH_{N d_C,0} + \frac{1}{C^N} \,J.
\]
Applying to~$f$, then applying~$\sigma$, yields,
appealing again to~\eqref{EQ:non-morph}, next again to Lemma~\ref{LM:tau-vspace}:
\begin{equation}\label{EQ:reduce-D-inside-sigma}
\begin{aligned}
\sigma(D_1^a D_2^b(f))
\in \frac{1}{\tau(C^N)} \; \sigma ( \mycH_{N d_C,0} (f) )
\subseteq \sum_{\bbeta \in B} \frac{K[s, \hx]_{\leq 2d_C N}}{s^{d_C N} \tau(C^N) } \sigma(\vD^\bbeta_x f) .
\end{aligned}
\end{equation}
Combining \eqref{EQ:move-D-inside-sigma} and~\eqref{EQ:reduce-D-inside-sigma} and using~$t = x_2$,
we obtain~\eqref{EQ:D-of-sigma}
where $D$ and~$A_N$ are set as in the lemma statement.
\end{proof}

\begin{lemma}\label{LEM:exists-P-alpha}
There exists a non-zero annihilator $P_\alpha(\hx; D_t)$ of $\Delta_{2,1}(f)$ satisfying~\eqref{EQ:bounds-on-P-alpha}.
\end{lemma}

\begin{proof}
Recall the definitions \eqref{EQ:def-V-N} and~\eqref{EQ:def-W-N} of $V_N$ and~$W_N$,
where $A_N$ and~$B$ are now fixed.
Lemma~\ref{LEM:D-sigma-in-vspace} has proved that
the $K(\hx)$-linear map defined by $\phi (P) := P (\sigma(f))$
is from~$V_N$ to~$W_N$.
Note that
\begin{equation}\label{EQ:dim-V-and-dim-W}
\dim_{K(\hx)}V_N = \binom{N+2}{2}, \quad  \dim_{K(\hx)}W_N \leq R (DN+1),
\end{equation}
where $R:= r_1 r_2 = O(r_f^2)$.
Fix
\begin{equation}\label{EQ:N}
N = 2 D R  = 4(d_1+d_2+1)\,r_1 r_2 = O(d_f r_f^2) ,
\end{equation}
so that
\begin{equation}\label{EQ:V-lt-W}
\dim_{K(\hx)} V_N - \dim_{K(\hx)} W_N = (3D-1) R + 1 > 0
\end{equation}
and $\phi$~is non-injective.
For all $i,j$ with $i+j \leq N$, by Lemma~\ref{LEM:D-sigma-in-vspace}
there exist polynomials $q_{\alpha, \bbeta}^{(i,j)} \in K[\hx]$
satisfying $\tdeg(q_{\alpha, \bbeta}^{(i,j)}) \leq DN$ and
\[
A_N(s,\hx) D_s^i D_t^j (\sigma(f)) = \sum_{\alpha \leq DN \atop \bbeta \in B} q_{\alpha, \bbeta}^{(i,j)} \, s^\alpha \sigma(\Dx^\bbeta f) \in W_N .
\]
A witness of non-injectivity will be provided
by polynomials $p_{i,j} \in K[\hx]$ such that
\[
\sum_{i + j \leq N} p_{i,j}(\hx) A_N(s,\hx) \, D_s^i D_t^j (\sigma(f)) = 0,
\]
that is, by coefficient extraction, such that for all $\alpha \leq DN$ and~$\bbeta \in B$,
\[
\sum_{i + j \leq N} p_{i,j} \, q_{\alpha, \bbeta}^{(i,j)} = 0.
\]
Hence we have a linear system
\begin{equation*}
\begin{pmatrix}
	\ldots\\
	\ldots q_{\alpha, \bbeta}^{(i,j)} \ldots \\
	\ldots
\end{pmatrix}
\begin{pmatrix}
	\vdots\\
	p_{i,j}\\
	\vdots
\end{pmatrix}
= 0 ,
\end{equation*}
where the polynomials $q_{\alpha, \bbeta}^{(i,j)}$ have total degree at most~$DN$.
This system has $\dim_{K(\hx)} W_N$~rows and $\dim_{K(\hx)} V_N$~columns,
where those dimensions are given by~\eqref{EQ:dim-V-and-dim-W},
and by the inequality~\eqref{EQ:V-lt-W}
it has more columns than rows.
So, Lemma~\ref{LM:deg} applies and
leads to a non-zero solution $(p_{i, j})$ satisfying
\begin{equation*}
\tdeg(p_{i, j}) \leq D N \times R (N D+1) = O(D^4 R^3) = O(d_f^4 r_f^6) ,
\end{equation*}
where we used~\eqref{EQ:N}.
The operator $P := \sum_{i+j \leq N} p_{i,j} \, D_s^i D_t^j$
satisfies~$P(\sigma(f)) = 0$
and can be written
\begin{equation*}
P = \sum_{i = \alpha}^{\beta} P_i(\hx; D_t) D_s^i
\end{equation*}
with $P_{\alpha} (\hx; D_t) \neq 0$.
Then $P_{\alpha}$ annihilates~$\Delta_{2,1}(f)$
and satisfies the announced bounds~\eqref{EQ:bounds-on-P-alpha}.
\end{proof}

\subsubsection{Controlling and combining the $D_s^iD_{x_h}^j(\sigma(f))$}
\label{SECT:Ds-Dx}

For each $h \in \{ 3,\dots, n \}$,
we proceed by an argument similar to the argument of Section~\ref{SECT:Ds-Dt}
to construct an operator $P_h \in K[\hx]\<D_{x_h}, D_s>$ such that $P_h(\sigma(f)) = 0$.
The proof is a bit simpler,
because the action of~$D_{x_h}$ on~$\sigma(f)$ is simpler
that the action of~$D_t$ on it.
This time, we consider
$B = \{0,1, \dots, r_1\} \times \{0,1, \dots, r_2\} \times \{0,1, \dots, r_h\}$,
the polynomial~$C$,
and $d_C = d_1+d_2+d_h\leq 3 d_f$
as set by Definition~\ref{DEF:C} for $S := \{1, 2, h\}$.
In analogy with \eqref{EQ:bounds-on-P-alpha} and~\eqref{EQ:bounds-on-P-h-alpha},
for each~$N \in \bN$, we introduce
\begin{equation}\label{EQ:def-V-N-h}
V_N = A_N(s,\hx) \ \span_{K(\hx)} \{   D_s^i D_{x_h}^j  \; | \;\, i+j \leq N\},
\end{equation}
where $ A_N = s^{(d_C+2) N } \tau(C^N)\in K[s,\hx] $, and
\begin{equation}\label{EQ:def-W-N-h}
W_N = \span_{K(\hx)} \{ s^\alpha \sigma(\Dx^\bbeta f) \; | \;\,  \alpha \leq DN, \bbeta \in B \} ,
\end{equation}
where $D = 2+2 d_C = O(d_f)$.
We will again prove that the map defined by $\phi (P) := P (\sigma(f))$
is $K(\hx)$-linear from~$V_N$ to~$W_N$.

In analogy with~\eqref{EQ:G}, define for any~$N \in \bN$:
\begin{equation}\label{EQ:G-h}
\cG_N := \bigoplus_{a+b+c \leq N} x_1^{-N} K[x_1, x_2]_{\leq N} D_1^a D_2^b D_{x_h}^c .
\end{equation}

\begin{lemma}\label{LEM:D-sigma-in-vspace-h}
Let $B$, $C$, and~$d_C$ be as defined at the beginning of Section~\ref{SECT:Ds-Dx},
that is, as set by Definition~\ref{DEF:C} for $S := \{1, 2, h\}$.
Then, if $i+j \leq N$, then
\begin{equation}\label{EQ:D-of-sigma-h}
D_s^i D_{x_h}^j (\sigma(f)) \in \sum_{\alpha \leq DN \atop \bbeta \in B} \frac{K[\hx]_{\leq D N}}{A_N(s,\hx)} s^\alpha \sigma(\vD^\bbeta_x f) .
\end{equation}
\end{lemma}

\begin{proof}
If $i+j \leq N$, then Lemma~\ref{LEM:D-sigma-in-G}, the definition~\eqref{EQ:G-h}, Equation~\eqref{EQ:non-morph} and Lemma~\ref{LM:tau-vspace} imply
\begin{equation}\label{EQ:move-D-inside-sigma-h}
\begin{aligned}
D_s^i D_{x_h}^j (\sigma(f)) &\in D_{x_h}^j \sigma (\cG_i(f)) \subseteq \sigma (\cG_{i+j}(f)) \subseteq \sigma (\cG_N(f)) \\
  &= \sum_{a+b+c \leq N} \tau(x_1^{-N} K[x_1,x_2]_{\leq N}) \, \sigma(D_1^a D_2^b D_{x_h}^c (f)) \\
  &\subseteq \sum_{a+b+c \leq N} \frac{K[s,t]_{\leq 2N}}{s^{2N}} \, \sigma(D_1^a D_2^b D_{x_h}^c(f)) .
\end{aligned}
\end{equation}
Next, by Lemma~\ref{LM:reduce-2} with $u := N \geq v:= a+b+c$ and~$t := 0$, we have
\[
D_1^a D_2^b D_{x_h}^c \in \mycH_{0, a+b+c} \subseteq \frac{1}{C^N} \, \mycH_{N d_C,0} + \frac{1}{C^N} \,J.
\]
Applying to~$f$, then applying~$\sigma$, yields,
appealing again to~\eqref{EQ:non-morph}, next again to Lemma~\ref{LM:tau-vspace}:
\begin{equation}\label{EQ:reduce-D-inside-sigma-h}
\begin{aligned}
\sigma(D_1^a D_2^b D_{x_h}^c(f))
\in \frac{1}{\tau(C^N)} \; \sigma ( \mycH_{N d_C,0} (f) )
\subseteq \sum_{\bbeta \in B} \frac{K[s, \hx]_{\leq 2d_C N}}{s^{d_C N} \tau(C^N) } \sigma(\vD^\bbeta_x f) .
\end{aligned}
\end{equation}
Combining \eqref{EQ:move-D-inside-sigma-h} and~\eqref{EQ:reduce-D-inside-sigma-h} and using~$t = x_2$,
we obtain~\eqref{EQ:D-of-sigma-h}
where $D$ and~$A_N$ are set as in the lemma statement.
\end{proof}

\begin{lemma}\label{LEM:exists-P-h-alpha}
There exists a non-zero annihilator $P_\alpha(\hx; D_{x_h})$ of $\Delta_{2,1}(f)$ satisfying~\eqref{EQ:bounds-on-P-h-alpha}.
\end{lemma}

\begin{proof}
Recall the definitions \eqref{EQ:def-V-N-h} and~\eqref{EQ:def-W-N-h} of $V_N$ and~$W_N$.
Lemma~\ref{LEM:D-sigma-in-vspace-h} has proved that
the $K(\hx)$-linear map defined by $\phi (P) := P (\sigma(f))$
is from~$V_N$ to~$W_N$.
Note that
\begin{equation*}
\dim_{K(\hx)}V_N = \binom{N+2}{2}, \quad  \dim_{K(\hx)}W_N \leq R (DN+1),
\end{equation*}
where $R:= r_1 r_2 f_h = O(r_f^3)$,
and fix
\begin{equation}\label{EQ:N-h}
N = 2DR = 4(d_1+d_2+d_h+1)\,r_1 r_2 r_h= O(d_f r_f^3) .
\end{equation}
The thast three formulas in terms of $R$ and~$D$ are the same as in Lemma~\ref{LEM:exists-P-h-alpha},
with only the values of $R$ and~$D$ changed,
so the inequality
\begin{equation}\label{EQ:V-lt-W-h}
\dim_{K(\hx)} V_N - \dim_{K(\hx)} W_N = (3D-1) R + 1 > 0
\end{equation}
holds again,
and $\phi$~is non-injective.
The proof by linear algebra continues
as in the proof of Lemma~\ref{LEM:exists-P-alpha},
recombining expressions $A_N \, D_s^i D_{x_h}^j (\sigma(f))$
instead of expressions $A_N \, D_s^i D_t^j (\sigma(f))$.
It constructs a non-zero operator
\begin{equation*}
P_h := \sum_{i+j \leq N} \, p_{i,j} \, D_s^i D_{x_h}^j
= \sum_{i = \alpha_h}^{\beta_h} P_{h,i}(\hx; D_{x_h}) D_s^i
\in K[\hx]\<D_s, D_{x_h}>
\end{equation*}
satisfying $P_h (\sigma(f)) = 0$,
$P_{h, \alpha_h} \neq 0$,
and
\begin{equation}\label{EQ:pij}
\tdeg (p_{i,j}) \leq N D \times R (N D+1) = O(D^4 R^3) = O(d_f^4 r_f^{9}) .
\end{equation}
Then $P_{h, \alpha_h}$ annihilates~$\Delta_{2,1}(f)$
and satisfies the announced bounds~\eqref{EQ:bounds-on-P-h-alpha}.
\end{proof}

\subsubsection{Iterating primary diagonals}
\label{sec:iterating-prim-diags}

We can now estimate bounds on the degree and order
of an annihilating operator for the complete diagonal of~$f$
obtained by successive primary diagonals.
In analogy with the definition~\eqref{eq:complete-diag} of the complete diagonal,
we consider the partial diagonal
\begin{equation*}
g := \Delta_{k+1,k}\Delta_{k,k-1} \cdots \Delta_{2,1}(f) \in K[[x_{k+1},\dots,x_n]] .
\end{equation*}
obtained after $k$~iterations of a primary diagonal.
Assume that there exists a non-zero annihilating operator for~$g$
with respective degree and order bounds
\begin{equation}\label{EQ:k-th-bound}
O\left(d_f^{u(k)} r_f^{v(k)}\right)
\quad\text{and}\quad
O\left( d_f^{s(k)} r_f^{t(k)}\right) .
\end{equation}
By Theorem~\ref{THM:primary bound} applied to~$f = g$,
there exists a non-zero annihilating operator for~$\Delta_{k+2,k+1}(g)$,
with respective degree and order bounds analogous to~\eqref{EQ:k-th-bound}
for exponents $u(k+1)$, $v(k+1)$, $s(k+1)$, $t(k+1)$ given by
\begin{equation*}
\begin{pmatrix}
    u(k+1) & v(k+1) \\
    s(k+1) & t(k+1)
\end{pmatrix}
=
\begin{pmatrix}
    4 & 9 \\
    1 & 3
\end{pmatrix}
\begin{pmatrix}
    u(k) & v(k) \\
    s(k) & t(k)
\end{pmatrix}.
\end{equation*}
Here, the entries of the constant matrix are obtained as the maximums of the exponents
appearing in the big~O terms in \eqref{EQ:bounds-on-P-alpha} and~\eqref{EQ:bounds-on-P-h-alpha}.
This sets up a recurrence that we proceed to analyze.
The matrix~$\bigl(\begin{smallmatrix}4 & 9 \\ 1 & 3\end{smallmatrix}\bigr)$
has two eigenvalues satisfying~$\lambda^2-7\lambda+3 = 0$, namely
\begin{equation}
\lambda_1 := \dfrac{7+\sqrt{37}}{2} \approx 6.54\dots , \quad
\lambda_2 := \dfrac{7-\sqrt{37}}{2} \approx 0.46\dots .
\end{equation}
Taking initial values for~$s,t,u,v$ in to account, we get
\begin{equation}\label{EQ:s-t-u-v}
\begin{aligned}
s(k) &= \frac{1}{\sqrt{37}}\lambda_1^k - \frac{1}{\sqrt{37}}\lambda_2^k
        & \approx (0.16\dots) \lambda_1^k - (0.16\dots) \lambda_2^k , \\
t(k) &= \left( \frac{1}{2}-\frac{1}{2\sqrt{37}} \right) \lambda_1^k +
        \left( \frac{1}{2}+\frac{1}{2\sqrt{37}} \right) \lambda_2^k
        & \approx (0.42\dots) \lambda_1^k +  (0.58\dots) \lambda_2^k , \\
u(k) &= \left( \frac{1}{2}-\frac{5}{2\sqrt{37}} \right) \lambda_1^k +
        \left( \frac{1}{2}+\frac{5}{2\sqrt{37}} \right) \lambda_2^k
        & \approx (0.09\dots) \lambda_1^k + (0.91\dots) \lambda_2^k , \\
v(k) &= \frac{9}{\sqrt{37}}\lambda_1^k - \frac{9}{\sqrt{37}}\lambda_2^k
        & \approx (1.48\dots) \lambda_1^k - (1.48\dots) \lambda_2^k .
\end{aligned}
\end{equation}
Degree and order bounds for an annihilating operator~$P$ of~$\Delta(f)$
are obtained for~$k = n - 1$,
and \eqref{EQ:s-t-u-v}~leads to the respective asymptotic formulas
\begin{equation*}
\begin{aligned}
\deg(P) &= O\left(d_f^{u(n-1)} r_f^{v(n-1)}\right) = d_f^{O(\lambda_1^n)} r_f^{O(\lambda_1^n)} , \\
\ord(P) &= O\left(d_f^{s(n-1)} r_f^{t(n-1)}\right) = d_f^{O(\lambda_1^n)} r_f^{O(\lambda_1^n)} .
\end{aligned}
\end{equation*}
when $n$, $d_f$, and~$r_f$ tend independently to infinity,
and where the constants in the big~$O$'s are small (at most~$1$).

\subsection{Complete diagonal in a single step}
\label{SECT:single-step}

Following \cite[Remarks, item~(3)]{Lipshitz1988}, instead of iterating primary diagonal transformations,
we can get the operator that annihilates the complete diagonal of~$f$ in a single step.
The goal of this subsection is indeed
the construction of a specific linear differential operator annihilating~$\Delta(f)$
that satisfies the bounds presented in the following theorem.
These bounds are simply exponential in~$n$,
and therefore asymptotically smaller than the bounds obtained by the method by iteration,
which are doubly exponential in~$n$.

\begin{theorem}\label{THM:complete bound}
Let $f\in K[[\vx]]$ be {D-finite} over $K(\vx)$ and let $ d_i, f_i, d_f, r_f $ be as in Definition~\ref{DEF:bound}.
Then, there exists an annihilating operator~$\tilde P$ of~$\Delta(f)$
in~$K[t]\<D_t>$ that satisfies, for all~$\varepsilon > 0$,
\begin{equation}\label{EQ:bounds-N-N'}
\begin{aligned}
\deg(\tilde P) \leq N' = O((2+\varepsilon)^n n^{2n} d_f^n r_f^n),
\ \
\ord(\tilde P) \leq N = O((2+\varepsilon)^n n^{2n-1} d_f^{n-1} r_f^n),
\end{aligned}
\end{equation}
when $n$, $d_f$, and~$r_f$ tend independently to infinity,
and where
\begin{equation}\label{EQ:N-and-N'}
N' = (2D+1)^n \prod_{j=1}^nr_j , \
N = \frac{(2D+1)^n}{D} \prod_{j=1}^nr_j ,
\ \text{for}\
D = n\biggl(2 + \sum_{i=1}^n d_i\biggr) .
\end{equation}
\end{theorem}

To prepare for the proof,
we specialize the setting introduced at the beginning of Section~\ref{SEC:bound}
by setting $m=n-1$, so that $t = x_n$,
and we define two maps $\sigma$ and~$\tau$ from~$\cS$ to~$M$ by
\begin{equation}\label{EQ:gen-tau}
\tau(h(\vx)) = h\left(s_1, \frac{s_2}{s_1}, \frac{s_3}{s_2}, \dots, \frac{s_{n-1}}{s_{n-2}}, \frac{t}{s_{n-1}}\right)
\quad\text{and}\quad
\sigma(h(\vx)) = \frac{\tau(h(\vx))}{s_1 \dotsm s_{n-1}} ,
\end{equation}
which the reader will compare with~\eqref{EQ:simple-tau}.
Hence, as in the previous subsection,
$\tau$~is a ring morphism and the formula~\eqref{EQ:non-morph} holds again.

In order to generalize Lemmas \ref{LEM:coeff-extraction-univ} and~\ref{LEM:ann-univ},
we introduce some convenient notation for coefficient extraction.
For a series
\begin{equation*}
g = \sum_{\vi,j} g_{\vi,j} \vs^\vi t^j \in M ,
\end{equation*}
variables $v_1,\dots,v_\ell$ and exponents $e_1,\dots,e_\ell$,
with $\{v_1,\dots,v_\ell\} \subset \{\vs,t\}$,
we denote by
\[ [v_1^{e_1}\dotsb v_\ell^{e_\ell}] g \]
the sub-series of~$g$ involving only the monomials~$\vs^\vi t^j$ in which
$v_1$~has exponent exactly~$e_1$, $v_2$~has exponent exactly~$e_2$, etc.
Note that this is mere notation and that
$[v_1^{e_1}]g$~need not be equal to~$[v_1^{e_1}v_2^0]g$
although~$v_1^{e_1} = v_1^{e_1}v_2^0$ in~$M$.
We do analogously with an operator
$P \in (K[t]\<D_t>)[\vD_s]$
and a set of variables $\{v_1,\dots,v_\ell\} \subset \{\vD_s\}$,
with the convention that coefficients are always written to the left of the monomials.

\begin{lemma}\label{LEM:coeff-extraction}
Let $P \in (K[t]\<D_t>)[\vD_s]$ be a non-zero operator
viewed with coefficients in~$K[t]\<D_t>$.
Consider any lexicographical order~$\succ$
on the commutative monoid generated by $\{D_{s_1}, \dots, D_{s_{n-1}}\}$,
e.g., the lexicographical order
for which $D_{s_1} \succ D_{s_2} \succ \cdots \succ D_{s_{n-1}}$.
Let $D_{s_1}^{\alpha_1} \dotsb D_{s_{n-1}}^{\alpha_{n-1}}$ be the minimal monomial in~$P$
with respect to this order, so that
\begin{equation}\label{EQ:def-P}
P = \tilde P(t; D_t) D_{s_1}^{\alpha_1} \dotsb D_{s_{n-1}}^{\alpha_{n-1}} + \text{terms with higher monomials}
\end{equation}
for some non-zero $\tilde P \in K[t]\<D_t>$.
Additionally, let
\begin{equation*}
g = \sum_{\vi,j} g_{\vi,j} \vs^\vi t^j
\end{equation*}
be any series in~$M$.
Then,
\begin{equation*}
[s_1^{-(\alpha_1+1)} \dotsb s_{n-1}^{-(\alpha_{n-1}+1)}] P(g) =
(-1)^{|\aalpha|} \alpha_1 ! \dotsb \alpha_{n-1} ! \, \tilde P([s_1^{-1}\dotsb s_{n-1}^{-1}] g) .
\end{equation*}
\end{lemma}

\begin{proof}
For the proof, we fix the lexicographical order~$\succ$
to satisfy $D_{s_1} \succ D_{s_2} \succ \cdots \succ D_{s_{n-1}}$.
Any other lexicographical order would be dealt with by obvious modifications.
We claim that, for any~$i$, after writing
\begin{equation*}
P = \bar P(t; D_t, D_{s_{i+1}}, \dots, D_{s_{n-1}}) D_{s_1}^{\alpha_1} \dotsb D_{s_i}^{\alpha_i} + Q
\end{equation*}
for some non-zero $\bar P \in K[t]\<D_t,D_{s_{i+1}},\dots,D_{s_{n-1}}>$
and some operator~$Q$ whose monomials~$\vD_s^\bbeta$ are all such that
$(\beta_1,\dots,\beta_i)$ is lexicographically higher than $(\alpha_1,\dots,\alpha_i)$,
we have
\begin{equation}\label{EQ:partial-extraction}
[s_1^{-(\alpha_1+1)} \dotsb s_i^{-(\alpha_i+1)}] P(g) =
(-1)^{\alpha_1+\dots+\alpha_i} \alpha_1 ! \dotsb \alpha_i ! \, \bar P([s_1^{-1}\dotsb s_i^{-1}] g) .
\end{equation}
The proof is by induction on~$i \in \{0,\dots,n-1\}$.
The base case~$i=0$ corresponds to no coefficient extraction and~$\bar P = P$,
so that \eqref{EQ:partial-extraction}~is the tautology~$P(g) = 1 \times \bar P(g)$.
Fix~$i\geq1$ and,
in order to prove~\eqref{EQ:partial-extraction},
assume the analog of~\eqref{EQ:partial-extraction} at~$i-1$,
that is,
\begin{equation}\label{EQ:partial-extraction-hyp-ind}
\begin{aligned}
[s_1^{-(\alpha_1+1)} &\dotsb s_{i-1}^{-(\alpha_{i-1}+1)}] P(g) = \\
& (-1)^{\alpha_1+\dots+\alpha_{i-1}} \alpha_1 ! \dotsb \alpha_{i-1} ! \, \hat P([s_1^{-1}\dotsb s_{i-1}^{-1}] g) ,
\end{aligned}
\end{equation}
for some non-zero
\begin{equation*}
\hat P = \sum_{j\geq\alpha_i} \hat P_j(t; D_t, D_{s_{i+1}}, \dots, D_{s_{n-1}}) D_{s_i}^j
\in K[t]\<D_t,D_{s_i},\dots,D_{s_{n-1}}> .
\end{equation*}
Consider a series~$c \in M$ involving only~$t,s_{i+1},\dots,s_{n-1}$,
as well as some integer~$u \in \bZ$, to compute
\begin{equation*}
[s_i^{-(\alpha_i+1)}] \hat P(c s_i^u) =
  \sum_{j\geq\alpha_i} \hat P_j(c) \, u (u-1) \dotsb (u-j+1) \, [s_i^{-(\alpha_i+1)}] s_i^{u-j} .
\end{equation*}
The last term $[s_i^{-(\alpha_i+1)}] s_i^{u-j}$ is equal to~$1$ if and only if~$j = u+\alpha_i+1$,
and is zero otherwise.
So the sum reduces to $\hat P_{u+\alpha_i+1}(c) \, u (u-1) \dotsb (-\alpha_i)$.
This is zero if~$u \geq 0$ because of the polynomial in~$u$, but also if~$u \leq -2$ because $\hat P_j = 0$ if~$j < \alpha_i$.
The only possibly non-zero case is therefore for~$u = -1$,
making the sum equal to~$(-1)^{\alpha_i} \alpha_i! \, \hat P_{\alpha_i}(c)$.
By linearity, we obtain
\begin{equation}\label{EQ:extract-just-one-var}
[s_i^{-(\alpha_i+1)}] \hat P([s_1^{-1}\dotsb s_{i-1}^{-1}] g) =
  (-1)^{\alpha_i} \alpha_i! \, \hat P_{\alpha_i}([s_i^{-1}] \, [s_1^{-1}\dotsb s_{i-1}^{-1}] g) .
\end{equation}
Applying~$[s_i^{-(\alpha_i+1)}]$ to~\eqref{EQ:partial-extraction-hyp-ind},
combining with~\eqref{EQ:extract-just-one-var},
and setting~$\bar P = \hat P_{\alpha_i}$,
we thus obtain~\eqref{EQ:partial-extraction}.
The case~$i = n-1$ proves the lemma by providing~$\tilde P = \bar P$.
\end{proof}

Consider again a non-necessarily D-finite series~$f$ as in~\eqref{EQ:def-f}.
By the definition~\eqref{eq:complete-diag} of the complete diagonal~$\Delta(f)$,
and by the definition~\eqref{EQ:gen-tau} of $\tau$ and~$\sigma$,
this complete diagonal~$\Delta(f)$ is~$[s_1^{-1}\dotsb s_{n-1}^{-1}]\sigma(f)$.
We will now derive the following analogue of Lemma~\ref{LEM:ann-univ}.

\begin{lemma}\label{LM:ann}
Let $f$ be as in~\eqref{EQ:def-f}.
Fix any lexicographical order~$\succ$
on the commutative monoid generated by $\{D_{s_1}, \dots, D_{s_{n-1}}\}$.
If $P(\sigma(f)) = 0$ for $P$ and~$\tilde P\neq0$ as in~\eqref{EQ:def-P},
then $\tilde P$~annihilates~$\Delta(f)$.
\end{lemma}
\begin{proof}
Lemma~\ref{LEM:coeff-extraction} and the equality
$[s_1^{-1}\dotsb s_{n-1}^{-1}]\sigma(f) = \Delta(f)$
imply
\begin{equation*}
(-1)^{|\aalpha|} \alpha_1 ! \dots \alpha_{n-1} ! \, \tilde P (\Delta(f)) =
[s_1^{-(\alpha_1+1)} \dotsb s_{n-1}^{-(\alpha_{n-1}+1)}]P(\sigma(f)) = 0 .
\end{equation*}
Hence, $\tilde P(\Delta(f)) = 0$.
\end{proof}

We will now construct an operator~$P$.
Henceforth, it will be convenient to write~$w$ in place of~$s_1 \dotsm s_{n-1}$
and $D_i$ in place of~$D_{x_i}$, for $ i = 1,\dots, n $.
Define
\[
\cG_m := \frac{K[\vs , t]_{\leq 2nm}}{w^{2m}} \sigma\left( \span_K \, \left\lbrace  \Dx^\aalpha f \mid |\aalpha|\leq m \right\rbrace  \right) .
\]
For convenience, write $s_0 := 1, s_n :=t$.
By the chain rule, for all $g \in \cS$ and each $i = 1,2, \dots, n-1$, we have
\begin{equation}\label{EQ:1.1}
D_{s_i}(\sigma(g)) = -\frac{1}{s_i} \sigma(g) + \frac{1}{s_{i-1}}\sigma(D_i \, (g)) - \frac{s_{i+1}}{s_i^2}\sigma(D_{i+1}(g)),
\end{equation}
and
\begin{equation}\label{EQ:1.2}
D_t(\sigma(g)) = \frac{1}{s_{n-1}} \sigma(D_{n}(g)).
\end{equation}
For all $|\aalpha|\leq m$, and all $p(\vs,t) \in K[\vs,t]_{\leq 2nm}$,
the chain rule implies that if $ 1\leq i \leq n-1 $, then
\begin{equation*}
\begin{aligned}
D_{s_i}
\left(\frac{p(\vs,t)}{w^{2m}} \,  \sigma(\vD_x^\aalpha f) \right)
& = D_{s_i}\left(\frac{1}{w^{2m}}\right) p \, \sigma(\vD_x^\aalpha f)
+ \frac{D_{s_i}(p)}{w^{2m}} \, \sigma(\vD_x^\aalpha f)
\\
& \qquad + \frac{p}{w^{2m}}  \, D_{s_i}(\sigma(\vD_x^\aalpha f)).
\end{aligned}
\end{equation*}
Rewriting the first two terms of the right-hand side over the denominator $w^{2(m+1)}$
shows that they are both in~$\cG_{m+1}$.
Similarly, making $g = \vD_x^\aalpha f$ in~\eqref{EQ:1.1}
and rewriting over the denominator~$w^{2(m+1)}$
shows that the third term is also in~$\cG_{m+1}$.
Therefore, $D_{s_i} \, \cG_m \subseteq \cG_{m+1}$.
A similar proof, using \eqref{EQ:1.2}, also shows
$D_t \, \cG_m \subseteq \cG_{m+1} $.
Since $1 \in \cG_0$, we get by induction that for all $\vi \in \bN^{n-1}$ and~$j \in \bN$,
\[
D_t^j \, \vD_{s}^{\vi} (\sigma(f)) \in \cG_{j+|\vi|}.
\]
Also note that $\cG_m \subseteq \cG_{m'}$ if~$m \leq m'$.
Now, if $ k \leq N',\, j + |\vi| \leq N  $, then
\begin{equation}\label{EQ:3}
t^k D_t^j \vD_{s}^{\vi}
(\sigma(f)) \in \frac{K[\vs , t]_{\leq 2nN+N'}}{w^{2N}} \, \sigma\left( \span_K \, \left\lbrace  \Dx^\aalpha f \mid |\aalpha|\leq N \right\rbrace  \right).
\end{equation}
Using Definition \ref{DEF:C} when $S = \{1, \dots,n \}$ fixes $B = \prod_{i=1}^n \left\lbrace 0, 1, \dots, r_i - 1\right\rbrace$, the polynomial~$C$, and $d_C = \sum_{j=1}^n d_j \leq n d_f$.
Then by Lemma \ref{LM:reduce-2}, with $u = N, v = |\aalpha|, t = 0$, we have
\[ \Dx^\aalpha\in \frac{1}{C^N}\mycH_{N d_C , 0} + \frac{1}{C^N} J. \]
Applying to $ f $, then applying $ \sigma $, yields:
\begin{equation}\label{EQ:4}
\sigma(\Dx^\aalpha f) \in \frac{1}{\tau(C^N)} \sigma(\mycH_{N d_C, 0} \, (f)) \subseteq \frac{K[\vs, t]_{ \leq nNd_C}}{w^{d_CN}\tau(C^N)} \bigoplus_{ \bbeta\in B} K \, \sigma  (\Dx^\bbeta f).
\end{equation}
Therefore, by \eqref{EQ:3} and~\eqref{EQ:4},
and for $D$~defined as in~\eqref{EQ:N-and-N'},
we have
\begin{equation}\label{EQ:5}
t^k D_t^j\vD_s^\vi \, \sigma(f) \in \frac{K[\vs, t]_{ \leq DN + N'}}{w^{(2+d_C)N} \tau (C^N)}  \bigoplus_{ \bbeta\in B} K \, \sigma  (\Dx^\bbeta f) .
\end{equation}
Denote $A_N(\vs,t) := w^{(2+d_C)N} \tau (C^N) \in K[\vs,t]$.
For any given $N'$ and~$N$, define
\[
V_{N,N'} = A_N(\vs,t) \ \span_K \{t^k D_t^j \vD_s^\vi \; | \; k \leq N',\, j + |\vi| \leq N\}
\]
and
\[
W_{N,N'} = \sum_{\bbeta\in B}  K[\vs, t]_{ \leq D N + N'}  \sigma(\Dx^\bbeta f) .
\]
We have proved by~\eqref{EQ:5} that there is a $K$-linear map~$\phi$ from~$V_{N,N'}$ to~$W_{N,N'}$ defined by $\phi (P) := P(\sigma(f))$.
Note that
\begin{equation}\label{EQ:dim-VNN-and-dim-WNN}
\dim_K V_{N,N'} = (N'+1) \binom{N+n}{n}, \quad  \dim_K W_{N,N'} \leq R \binom{DN+N'+n}{n} ,
\end{equation}
where $R := r_1 \cdots r_n = O(r_f^n)$.
Fix $N$ and~$N'$ as in~\eqref{EQ:N-and-N'}
($D$~has already been defined as there),
so that $N' = DN$, $N > R \frac{1 + 2nD}{D} > 2n > n$, and
\begin{equation*}
N^n N' = R (2ND + N)^n > R (2ND + n)^n = R (DN + N' + n)^n ,
\end{equation*}
from which follows, with the help of~\eqref{EQ:dim-VNN-and-dim-WNN},
\begin{equation*}
\begin{aligned}
\dim_K V_{N,N'} &> N' \frac{N^n}{n!} \\
  &> R \frac{(DN + N' +n )^n}{n!}  > R \binom{DN+N'+n}{n} \geq  \dim_K W_{N,N'} .
\end{aligned}
\end{equation*}
We therefore obtain $\dim_K V_{N,N'} - \dim_K W_{N,N'} > 0$,
so that $\phi$~is non-injective.
Consider any non-zero kernel element~$Z$, that is,
any family of constants $c_{\vi,j,k} \in K$
indexed by $\vi,j,k$ with $|\vi| + j \leq N$ and~$k \leq N'$,
and such that $\phi(Z) = 0$ for
\begin{equation}\label{eq:def-Z}
Z = \sum_{\vi + j \leq N, \ k \leq N'} c_{\vi,j,k} A_N \, t^k D_t^j \vD_s^\vi .
\end{equation}
Then, the operator $P := A_N^{-1} Z = \sum c_{\vi,j,k} \, t^k D_t^j \vD_s^\vi$
satisfies~$P(\sigma(f)) = 0$ as well.
From~\eqref{eq:def-Z} it follows that
\begin{equation}\label{EQ:pre-bounds-N-N'}
\deg(P) \leq N', \quad \ord(P) \leq N .
\end{equation}
Finally, $P$~can be written
\begin{equation*}
P = \tilde P(t; D_t) D_{s_1}^{\aalpha_1} \dots D_{s_{n-1}}^{\aalpha_{n-1}} + \text{higher terms}
\end{equation*}
with $\tilde P (t; D_t) \neq 0$,
and the operator~$\tilde P$ annihilates~$\Delta(f)$ by Lemma~\ref{LM:ann}
and satisfies the announced bounds~\eqref{EQ:bounds-N-N'}
because of~\eqref{EQ:pre-bounds-N-N'}.

Finishing the proof of Theorem~\ref{THM:complete bound} only requires
to validate the asymptotic estimates in~\eqref{EQ:bounds-N-N'}.
Set~$S := \sum_{i=1}^n d_i$, which goes to infinity because $d_f \leq S \leq n d_f$.
Fix~$\varepsilon > 0$.
From the value of~$D$ in~\eqref{EQ:N-and-N'} follow,
at least for $n \geq 1/(4\varepsilon)$,
\begin{equation*}
\begin{aligned}
D &= nS\left(1+\frac2S\right) \leq nS\left(1+\frac2{d_f}\right) = O(nS) , \\
2D+1 &= 2nS\left(1+\frac2S+\frac1{4nS}\right) \leq 2nS\left(1+\frac2{d_f}+\frac1{4nd_f}\right)
    \leq 2nS\left(1+\frac{2+\varepsilon}{d_f}\right) ,
\end{aligned}
\end{equation*}
and then, at least for $n \geq 1/(4\varepsilon)$ and~$d_f \geq 2$,
\begin{equation*}
\begin{aligned}
(2D+1)^n &\leq 2^nn^nS^n(1+\varepsilon/2)^n \leq (2+\varepsilon)^n n^{2n} d_f^n , \\
\frac{(2D+1)^n}{D} &\leq \frac{(2+\varepsilon)^n n^{2n} d_f^n}{n(2+S)}
    \leq \frac{(2+\varepsilon)^n n^{2n} d_f^n}{n(2+d_f)}
    \leq (2+\varepsilon)^n n^{2n-1} d_f^{n-1} .
\end{aligned}
\end{equation*}
Combining with~\eqref{EQ:N-and-N'} yields~\eqref{EQ:bounds-N-N'}.

\bigskip

\noindent {\bf Acknowledgement.}
The authors would like to thank Florent Koechlin and Pierre Lairez
for making us aware of several works related to diagonals of rational functions.

Shaoshi Chen and Pingchuan Ma were partially supported
by the National Key R\&D Program of China (No.\ 2023YFA1009401),
the NSFC grants (No.\ 12271511 and No.\ 11688101),
the CAS Funds of the Youth Innovation Promotion Association (No.\ Y2022001),
and the Strategic Priority Research Program of the Chinese Academy of Sciences (No.\ XDB0510201).
Frédéric Chyzak was supported in part
by the French ANR grant \href{https://mathexp.eu/DeRerumNatura/}{\emph{De rerum natura}} (ANR-19-CE40-0018),
by the French-Austrian ANR-FWF grant EAGLES (ANR-22-CE91-0007 \& FWF-I-6130-N),
and by the European Research Council under the European Union's Horizon Europe research and innovation programme, grant agreement 101040794 (10000 DIGITS).
Part of Pingchuan Ma's work was done during a 6-month visit at Inria in 2023.
Chaochao Zhu was partially supported by 2021 Research Start-up Fund for High-level Talents (Natural Sciences) (Project No.\ 00701092242).
All authors were also supported by the International Partnership Program of Chinese Academy of Sciences (Grant No.\ 167GJHZ2023001FN).

\bibliographystyle{plain}
\bibliography{diagonal.bib}

\def\cdprime{$''$}
\begin{thebibliography}{10}

\bibitem{AbdelazizKoutschanMaillard-2020-CC}
Y.~Abdelaziz, C.~Koutschan, and J.-M. Maillard.
\newblock On {C}hristol's conjecture.
\newblock {\em Journal of Physics A: Mathematical and Theoretical},
  53(20):205201, April 2020.

\bibitem{AdamczewskiBell-2013-DRA}
Boris Adamczewski and Jason~P. Bell.
\newblock Diagonalization and rationalization of algebraic {L}aurent series.
\newblock {\em Ann. Sci. Éc. Norm. Supér. (4)}, 46(6):963--1004, 2013.

\bibitem{AdamczewskiBostanCaruso-2023-SMC}
Boris Adamczewski, Alin Bostan, and Xavier Caruso.
\newblock A sharper multivariate {C}hristol's theorem with applications to
  diagonals and {H}adamard products, 2023.

\bibitem{BostanBoukraaMaillardWeil-2015-DRF}
A.~Bostan, S.~Boukraa, J.-M. Maillard, and J.-A. Weil.
\newblock Diagonals of rational functions and selected differential {G}alois
  groups.
\newblock {\em J. Phys. A}, 48(50):504001, 29, 2015.

\bibitem{BostanCarayolKoechlinNicaud-2020-WUP}
Alin Bostan, Arnaud Carayol, Florent Koechlin, and Cyril Nicaud.
\newblock Weakly-unambiguous {P}arikh automata and their link to holonomic
  series.
\newblock In {\em 47th International Colloquium on Automata, Languages, and
  Programming (ICALP 2020)}, volume 168 of {\em LIPIcs}, pages 114:1--114:16,
  2020.

\bibitem{BostanCarayolKoechlinNicaud-2025-WUP}
Alin Bostan, Arnaud Carayol, Florent Koechlin, and Cyril Nicaud.
\newblock Weakly-unambiguous {P}arikh automata and their link to holonomic
  series.
\newblock Long version of \cite{BostanCarayolKoechlinNicaud-2020-WUP},
  arXiv:2512.09823, December 2025.

\bibitem{BostanChyzakHoeijPech-2011-EFG}
Alin Bostan, Frédéric Chyzak, Mark van Hoeij, and Lucien Pech.
\newblock Explicit formula for the generating series of diagonal {3D} rook
  paths.
\newblock {\em Séminaire Lotharingien de Combinatoire}, 66(B66a):1--27, 2011.

\bibitem{BostanDumontSalvy-2017-ADW}
Alin Bostan, Louis Dumont, and Bruno Salvy.
\newblock Algebraic diagonals and walks: Algorithms, bounds, complexity.
\newblock {\em Journal of Symbolic Computation}, 83:68--92, 2017.
\newblock Special issue on the conference ISSAC 2015: Symbolic computation and
  computer algebra.

\bibitem{BostanLairezSalvy-2013-CTR}
Alin Bostan, Pierre Lairez, and Bruno Salvy.
\newblock Creative telescoping for rational functions using the
  {G}riffiths--{D}work method.
\newblock In {\em ISSAC'13}, pages 93--100. ACM, New York, 2013.

\bibitem{Christol-1984-DFR}
Gilles Christol.
\newblock Diagonales de fractions rationnelles et équations différentielles.
\newblock In {\em Study group on ultrametric analysis, 10th year: 1982/83,
  {N}o. 2}, pages Exp. No. 18, 10. Inst. Henri Poincaré, Paris, 1984.

\bibitem{Christol-1985-DFR}
Gilles Christol.
\newblock Diagonales de fractions rationnelles et équations de
  {P}icard-{F}uchs.
\newblock In {\em Study group on ultrametric analysis, 12th year, 1984/85,
  {N}o.\ 1}, pages Exp. No. 13, 12. Secrétariat Math., Paris, 1985.

\bibitem{Christol-1987-FHB}
Gilles Christol.
\newblock Fonctions hypergéométriques bornées.
\newblock {\em Groupe de travail d'analyse ultramétrique}, 14:1--16,
  1986--1987.

\bibitem{Christol-1988-DFR}
Gilles Christol.
\newblock Diagonales de fractions rationnelles.
\newblock In {\em Séminaire de {T}héorie des {N}ombres, {P}aris 1986--87},
  volume~75 of {\em Progr. Math.}, pages 65--90. Birkh\"auser Boston, Boston,
  MA, 1988.

\bibitem{Christol-2015-DRF}
Gilles Christol.
\newblock Diagonals of rational fractions.
\newblock {\em Eur. Math. Soc. Newsl.}, 97:37--43, 2015.

\bibitem{Coutinho1995}
S.~C. Coutinho.
\newblock {\em A primer of algebraic {$D$}-modules}, volume~33 of {\em London
  Mathematical Society Student Texts}.
\newblock Cambridge University Press, Cambridge, 1995.

\bibitem{Deligne-1984-ICE}
P.~Deligne.
\newblock Intégration sur un cycle évanescent.
\newblock {\em Inventiones Mathematicae}, 76:129--144, 1984.

\bibitem{Furstenberg-1967-AFF}
Harry Furstenberg.
\newblock Algebraic functions over finite fields.
\newblock {\em J. Algebra}, 7:271--277, 1967.

\bibitem{Gessel1981}
Ira Gessel.
\newblock Two theorems on rational power series.
\newblock {\em Utilitas Mathematica}, 19:247--254, 1981.

\bibitem{Kauers-2014-BDC}
Manuel Kauers.
\newblock Bounds for {D}-finite closure properties.
\newblock In {\em I{SSAC} 2014---{P}roceedings of the 39th {I}nternational
  {S}ymposium on {S}ymbolic and {A}lgebraic {C}omputation}, pages 288--295.
  ACM, New York, 2014.

\bibitem{Lipshitz1988}
Leonard~M. Lipshitz.
\newblock The diagonal of a {D}-finite power series is {D}-finite.
\newblock {\em J. Algebra}, 113(2):373--378, 1988.

\bibitem{Lipshitz1989}
Leonard~M. Lipshitz.
\newblock {D}-finite power series.
\newblock {\em J. Algebra}, 122(2):353--373, 1989.

\bibitem{Melczer-2021-IAC}
Stephen Melczer.
\newblock {\em An invitation to analytic combinatorics: from one to several
  variables}.
\newblock Texts and Monographs in Symbolic Computation. Springer, 2021.

\bibitem{Mishna-2020-ACM}
Marni Mishna.
\newblock {\em Analytic combinatorics: a multidimensional approach}.
\newblock Discrete Mathematics and its Applications. CRC Press, Boca Raton, FL,
  2020.

\bibitem{PemantleWilson-2008-TCE}
Robin Pemantle and Mark~C. Wilson.
\newblock Twenty combinatorial examples of asymptotics derived from
  multivariate generating functions.
\newblock {\em SIAM Rev.}, 50(2):199--272, 2008.

\bibitem{PemantleWilson-2024-ACSV}
Robin Pemantle, Mark~C. Wilson, and Stephen Melczer.
\newblock {\em Analytic combinatorics in several variables}, volume 212 of {\em
  Cambridge Studies in Advanced Mathematics}.
\newblock Cambridge University Press, Cambridge, second edition, 2024.

\bibitem{Salvy-2019-LDE}
Bruno Salvy.
\newblock Linear differential equations as a data structure.
\newblock {\em Found. Comput. Math.}, 19(5):1071--1112, 2019.

\bibitem{Stanley1980}
Richard~P. Stanley.
\newblock Differentiably finite power series.
\newblock {\em European J. Combin.}, 1(2):175--188, 1980.

\bibitem{Stanley-1999-EC2}
Richard~P. Stanley.
\newblock {\em Enumerative combinatorics. {V}ol. 2}, volume~62 of {\em
  Cambridge Studies in Advanced Mathematics}.
\newblock Cambridge University Press, Cambridge, 1999.

\bibitem{WuChen-2013-NDT}
Xiao~Li Wu and Shao~Shi Chen.
\newblock A note on the diagonal theorem for bivariate rational formal power
  series.
\newblock {\em Acta Math. Sinica (Chinese Ser.)}, 56(2):203--210, 2013.

\bibitem{Zeilberger1980}
Doron Zeilberger.
\newblock Sister celine's technique and its generalization.
\newblock {\em Journal of Mathematical Analysis and its Applications},
  85:114--145, 1982.

\bibitem{Zeilberger1990}
Doron Zeilberger.
\newblock A holonomic systems approach to special functions identities.
\newblock {\em J. Comput. Appl. Math.}, 32(3):321--368, 1990.

\end{thebibliography}

\end{document}